\documentclass[a4paper]{amsart}

\usepackage[latin1]{inputenc}
\usepackage{amsfonts}
\usepackage{amssymb}
\usepackage[leqno]{amsmath}
\usepackage{amsthm}
\usepackage{amscd}
\usepackage[all]{xypic}
\usepackage{enumerate}
\usepackage[mathscr]{eucal}
\usepackage{ulem}            
\usepackage{verbatim}           

\usepackage{times}

\newtheorem{theorem}{Theorem}[section]
\newtheorem{proposition}[theorem]{Proposition}
\newtheorem{corollary}[theorem]{Corollary}

\newtheorem{lemma}[theorem]{Lemma}

\theoremstyle{definition}
\newtheorem{definition}[theorem]{Definition}
\newtheorem{remark}[theorem]{Remark}
\newtheorem{example}[theorem]{Example}

\DeclareMathOperator{\ev}{ev}
\DeclareMathOperator{\id}{id}
\DeclareMathOperator{\Hom}{Hom}
     
\DeclareMathOperator{\Id}{Id}

\DeclareMathOperator{\im}{Im}
\DeclareMathOperator{\Spec}{Spec}
\DeclareMathOperator{\spec}{spec}   
\DeclareMathOperator{\rk}{rk}       
\DeclareMathOperator{\Gr}{Gr}      
\DeclareMathOperator{\Sl}{Sl}

\def\top{\textup{top}}

\def\N{\mathbb{N}}
\def\Z{\mathbb{Z}}
\def\R{\mathbb{R}}

\def\C{\mathbb{C}}

\def\O{\mathcal{O}}
\def\FDS{\mathcal{F}_{D\to S}}              
\def\FDCC{\mathcal{F}_{D\to CC}}            
\def\FCCS{\mathcal{F}_{CC\to S}}
\def\FSCC{\mathcal{F}_{S\to CC}}
\def\ot{\otimes}
\def\vphi{\varphi}
\def\F{\mathbb{F}}
\def\lto{\longrightarrow}

\def\lmto{\longmapsto}
\def\Fun{{\F_1}}                             
\def\Gm{\mathbb{G}_m}
\def\A{\mathbb{A}}
\def\otz{\ot_{\Z}}
\def\otf{\ot_{\Fun}}
\def\du{\amalg} 
\def\bdu{\coprod} 
\def\decomp{\overset{\circ}{\coprod}}
\def\Sets{\textrm{$\mathcal{S}$ets}}

\newcommand{\abs}[1]{\left|#1\right|} 
\newcommand{\card}[1]{\#\left(#1\right)} 
\newcommand{\dtext}[1]{\emph{\textbf{#1}}} 

\title{Torified varieties and their geometries over $\Fun$}
\author{Javier L\'{o}pez Pe\~na}
\address{Mathematics Research Centre\\ 
Queen Mary University of London\\
Mile End Road, London E1 4NS, United Kingdom}
\email{jlopez@maths.qmul.ac.uk \textrm{(J. L\'opez Pe\~na)}}
\thanks{During the ellaboration of this paper, J. L\'opez Pe\~na was supported initially by the Max-Planck-Institut f\"ur Mathematik in Bonn, and in the final stages by the EU Marie-Curie fellowship PIEF-GA-2008-221519 at Queen Mary University of London.}
\author{Oliver Lorscheid}
\address{Max-Planck Institut f\"ur Mathematik\\
Vivatsga\ss{}e, 7. D-53111, Bonn, Germany}
\email{oliver@mpim-bonn.mpg.de \textrm{(O. Lorscheid)}}
\thanks{O. Lorscheid was supported by the Max-Planck-Institut f\"ur Mathematik in Bonn.}

\begin{document}

\begin{abstract}
 This paper invents the notion of torified varieties: 
 A torification of a scheme is a decomposition of the scheme into split tori.
 A torified variety is a reduced scheme of finite type over $\Z$ endowed with a torification.
 Toric varieties, split Chevalley schemes and flag varieties are examples of this type of scheme.
 Given a torified variety whose torification is compatible with an affine open covering, 
 we construct a gadget in the sense of Connes-Consani and an object in the sense of Soul\'e
 and show that both are varieties over $\F_1$ in the corresponding notion.
 Since toric varieties and split Chevalley schemes satisfy the compatibility condition, 
 we shed new light on all examples of varieties over $\F_1$ in the literature so far.
 Furthermore, we compare Connes-Consani's geometry, Soul\'e's geometry and Deitmar's geometry,
 and we discuss to what extent Chevalley groups can be realized as group objects over $\F_1$ in the given categories.
\end{abstract}

\maketitle

\tableofcontents


\section*{Introduction}

A study seminar on $\Fun$,
which was held at the Max Planck Institute for Mathematics in Bonn in fall 2008, led to several discussions about the possibilities and limitations
of the various notions of geometries over $\F_1$ that were produced in recent years.
This paper subsumes the most relevant thoughts of those discussions. 
It was possible to establish a good part of varieties over $\F_1$ in the notion of Soul\'e, which was further developed by Connes and Consani.
While the philosopher's stone regarding $\F_1$-geometries is not found yet,
there will be many examples and remarks disclosing problems of the recent theories and hinting at directions one might try to go.

The idea of constructing objects over a ``field with one element'' goes back to Tits in \cite{Tits1957}, where the question about the interpretation of
Weyl groups as ``Chevalley groups over $\Fun$'' is posed. In recent years, a number of papers (\cite{Manin1995}, \cite{Soule2004}, \cite{Deitmar2005},
\cite{Deitmar2006}, \cite{Deitmar2007}, \cite{Toen2008}, \cite{Connes2008a}, \cite{Connes2008}, \cite{Manin2008}, \cite{Marcolli2009}, \cite{Connes2009}, \ldots) on the topic have appeared,
dealing mostly with the problem of defining a suitable notion of algebraic geometry over such an elusive object. 
Several non equivalent approaches have been tried, 
for instance Durov (cf. \cite{Durov2007}) and Shai-Haran (cf. \cite{ShaiHaran2007}) enlarged the category of schemes to obtain the spectrum of $\F_1$
in place of $\Spec \Z$ as final object,
Deitmar mimicked scheme theory using monoids (i.e.\ commutative semi-groups with $1$) in the place of commutative rings (cf. \cite{Deitmar2005,
Deitmar2006, Deitmar2007}), whereas To\"en and Vaqui\'e (cf. \cite{Toen2008}) described a categorical approach in terms of functors on monoids. There is also a more recent approach by Connes and Consani (cf. \cite{Connes2009}) combining these viewpoints.

Soul\'e proposed in \cite{Soule2004} that varieties over $\Fun$ should be functors that admit a base extension to $\Z$. He gave a precise realization by considering functors from the category of flat rings of finite type over $\Z$ to the category of finite sets
together with an evaluation, i.e.\ a natural transformation from this functor to the functor of homomorphisms from a fixed complex algebra to
the complexification of the given ring. Soul\'e showed that smooth toric varieties admit a model over $\F_1$ in his notion.
This approach was further developed by Connes and Consani in \cite{Connes2008} by exchanging flat finite rings by finite abelian groups
and doing some further refinements. They mention that Soul\'e's method of establishing smooth toric varieties over $\F_1$ still works
and they demonstrate this in the case of the multiplicative group scheme, affine space and projective space. 
However, their focus is on Chevalley schemes. 
To be precise, Connes and Consani establish split Chevalley schemes as varieties over $\F_{1^2}$.

In the present work, we generalize methods to show that all reduced schemes of finite type over $\Z$ 
that admit a decomposition by algebraic tori, dubbed \dtext{torified varieties}, 
have a model over $\F_1$ in both Soul\'e's and Connes-Consani's notion--provided they admit an open affine cover compatible with the decomposition.
This class of schemes includes toric varieties and split Chevalley schemes, which covers all examples in the literature so far.
Grassmannians and flag varieties are torified varieties as well,
but in general, they lack the extra condition of having a compatible atlas, which is necessary to define the base extension to $\Z$
in the given notions. 
However, the class of torified varieties could be a leading example for the development of new notions of geometries over $\F_1$.

Furthermore, we connect Deitmar's viewpoint (\cite{Deitmar2005}) with the previous. 
Namely, we construct an embedding of Deitmar's category of schemes over $\F_1$
that base extend to integral schemes of finite type over $\Z$ into the category of varieties over $\F_1$.
We also compare the two notions of varieties over $\F_1$, which seem to produce similar theories
except for one remarkable difference: Chevalley groups are more likely to be a variety over $\F_1$ after Soul\'e  
than they are after Connes-Consani (see Remark \ref{remark_S_result_on_group}). 
We show, however, that $\Sl(2)$ cannot be established as a group object in either notion.

The paper is organized as follows. In section \ref{section:torified}, we introduce the notion of \dtext{torification} of a scheme $X$ as a finite
family of immersions $\{\vphi_i:T_i \hookrightarrow X\}$ such that every $T_i$ is a split torus over $\Z$ 
and every geometric point of $X$ factorizes through
exactly one of such immersions. We consider schemes with torification together with morphisms that respect the torifications, 
called \dtext{torified morphisms}. We describe the zeta function of a torified variety over $\F_1$ and provide a list of examples of torified varieties.

In section \ref{section_CC-gadgets}, we recall the notion of Connes-Consani's gadgets and varieties over $\Fun$, and show how to associate a gadget
$\mathcal{L}(X,T)$ to every torified variety $X$ endowed with a torification $T$. We prove in Theorem \ref{thm_ll} that this gadget is actually a variety
providing an $\Fun$-model for $X$ whenever the torification is compatible with an affine open cover.
In particular, this result extends the one by Connes and Consani by realizing split Chevalley schemes over $\Fun$ (and not only over $\mathbb{F}_{1^2}$).

In section \ref{section:soule}, we recall Soul\'e's approach to $\Fun$-geometry. 
We show in Theorem \ref{thm_javier} that the previous result (Theorem \ref{thm_ll}), mutatis mutandis, also holds in this case.

In section \ref{section:deitmar}, we recall the notion of Deitmar's schemes over $\Fun$,
and refine the equivalence between the category of toric varieties and the category of schemes over $\F_1$ that base extend to 
connected integral schemes of finite type over $\Z$.

In section \ref{section:geometries}, we compare the three aforementioned notions of geometries over the field with one element by establishing 
functors between them. Deitmar's theory can be embedded into both the theory of Soul\'e and the theory of Connes and Consani.
There are further several ways to go from Connes-Consani's world to Soul\'e's world and back, 
but it is not clear if they compare one-to-one as we discuss in section \ref{subsection:CC_to_S}.
We summarize these results in the diagram of Theorem \ref{thm_large_diagram}.

We conclude the paper with remarks showing the boundaries of Soul\'e's and Connes-Consani's geometries, mainly the impossibility of obtaining the group
operation of Chevalley schemes as a morphism over $\Fun$. 
Further we recollect some thoughts that might eventually lead to new approaches to $\F_1$-geometries in future works.

\textbf{Acknowledgments:}
The authors thank all people that participated in the $\F_1$-study seminar, in particular Peter Arndt, Pierre-Emmanuel Chaput, Bram Mesland and Fr\'ed\'eric Paugam for giving interesting lectures at the seminar and participating on stimulating discussions. The authors thank Bas Edixhoven for his help on improving some proofs and Markus Reineke for providing an interesting counter example. The authors thank the Max-Planck Institut f\"ur Mathematik in Bonn for support and hospitality and for providing excellent working conditions.


\section{Torified varieties}\label{section:torified}

\subsection{The category of torified schemes}

 In this section, we will establish the definition of torified schemes and show some basic properties. If $X$ and $S$ are schemes over $\Z$, we will denote by $X(S):=\Hom(S,X)$ the set of $S$--points of $X$. The underlying topological space of $X$ will be denoted by $X^{top}$, its structure sheaf by $\O_X$, and the stalk at a point $x\in X^{top}$ by $\O_{X,x}$. Following \cite[\S I-4.2.1]{EGAI} by an immersion of schemes $f:Y\to X$ we will mean a morphism of schemes that factorizes as $Y\overset{g}\to Z \hookrightarrow X$, where $Z$ is a locally closed subscheme of $X$ and $g$ is an isomorphism.
 
 \begin{definition}
  Given a scheme $X$, a \dtext{decomposition} of $X$ consists of a family $\{Y_i \}_{i\in I}$ of locally closed (nonempty) subschemes $Y_i$ of $X$ such that for every algebraically closed field $\Omega$ one has
  \[
   \bdu_{i\in I} Y_i(\Omega) =  X(\Omega),
  \]
 or equivalently as a family of locally closed subschemes such that one has the equality
  \[
   \bdu_{i\in I}\abs{Y_i} =  \abs{X}.
  \]
 \end{definition}
 If this is the case, we will write for short $X=\decomp Y_i$. This property implies the following result:
 
 \begin{lemma}\label{lemma:decomp}
  Let $X=\decomp Y_i$ be a decomposition of the scheme $X$, and let $S$ be a scheme over $\Z$; then the map $\bdu Y_i(S) \to X(S)$ is injective. Moreover, if $S=\Spec k$ for a field $k$, it is a bijection.
 \end{lemma}
 
 \begin{proof}
  Denote by $\tau_i$ the natural immersion of $Y_i$ inside $X$. Assume there are $\alpha\in Y_i(S)$ and $\beta\in Y_j(S)$ mapping to the same element of $X(S)$. Is $S\neq \varnothing$, pick a geometric point $p:\Spec \Omega\to S$ of $S$. One has the commutative diagram
  \[
   \xymatrix@R=.5pc{
     & & Y_i \ar[dr]^{\tau_i} \\
    \Spec \Omega \ar@{..>}[r]^p & S \ar[ur]^{\alpha} \ar[dr]_{\beta}  & & X. \\
    & & Y_j \ar[ur]_{\tau_j} 
   }
  \]
  By the definition of a decomposition, the commutativity of the diagram implies $i=j$. Since $\tau_i=\tau_j$ is an immersion, it follows that $\alpha=\beta$, and so the injectivity of the map $\bdu Y_i(S) \to X(S)$.

  If $k$ is a field and $\alpha:\Spec k\to X$ a morphism, choose an algebraic closure $\Omega$ of $k$. The induced map $\Spec\Omega$ factors uniquely over one $Y_i$. As $Y_i\to X$ is an immersion, $\alpha$ also factors uniquely over $Y_i$.
 \end{proof}

If $X=\decomp_{i\in I} Y_i$ is a decomposition of $X$, we will consider the subset 
 \[
  I^o:=\{i\in I|\ \vphi_i\ \text{is an open immersion}\}.
 \]

	\begin{lemma}\label{lemma:torificationpoints}
		Let $X=\decomp_{i\in I} Y_i$. The following properties hold true:
		\begin{enumerate}
			\item The map $\bdu_{i\in I}Y_i^{top} \lto X^{top}$ is a continuous bijection.
			\item The cardinality of $I^o$ is bounded by the number of irreducible components of $X$. If $Y_i$ is irreducible for every $i\in I^o$, then $I^o$ stays in bijection with the irreducible components of $X$.
		\end{enumerate}
	\end{lemma}
	\begin{proof}\hfill\\
		\noindent(1): This follows from the universal property of the decomposition, taking into account that every point of $X^{top}$ is the image of some geometric point $\Spec \Omega \to X$.
		
		\noindent(2): If $Y_i\to X$ is an open immersion, then its image contains at least one generic point of $X$. It contains precisely one generic point when $Y_i$ is irreducible. Since the generic points of $X$ characterize the irreducible components of $X$, the lemma follows.
	\end{proof}

\begin{corollary}
	 If $X=\decomp_{i\in I} Y_i$ is a scheme of finite type over $\Z$, then $I$ is a finite set. 
\end{corollary}

\begin{proof}
	The proof is by induction on the dimension $n$ of $X$. If $n=0$, then $X^\top$ is a discrete space consisting of a finite number of points, and the claim of the lemma is immediate. 
 
	If $n>0$, then $X$ has a finite number of irreducible components. By the previous lemma, $I^o$ is a finite set, and the image of $(\coprod_{i\in I^o} Y_i)^\top$ in $X^\top$ is a dense open subset. Thus the image of $\coprod_{i\in (I-I^o)} Y_i$ in $X$ defines a closed subscheme of $X$, which is of dimension smaller than $n$. By the induction hypothesis, $I-I^o$ is finite and therefore $I$ is so.
\end{proof}

	\begin{definition}
		A scheme $X$ is \dtext{torifiable} if it has a decomposition $X=\decomp_{i\in I} T_i$, where for each $i\in I$ we have $T_i$ isomorphic to $\Gm^{d_i}$ (as algebraic groups) for $d_i\in \N$. In this case we will say that $T=\{\vphi_i:T_i\hookrightarrow X\}$ is a \dtext{torification} of $X$, and call the couple $(X,T)$ a \dtext{torified scheme}. A \dtext{torified variety} is a torified scheme that is reduced and of finite type over $\Z$. A torification $X=\decomp_{i\in I} T_i$ is \dtext{affine} if there is an affine open cover $\{ U_j\}$ of $X$ respecting the torification, i.e.\ for each $j$ there is a subset $I_j\subseteq I$ satisfying that $U_j = \decomp_{i\in I_j} T_i$. 	
	\end{definition}
	
	We will denote by $(X,T)$ the scheme $X$ with a fixed torification $T$ when needed, though often we will denote $(X,T)$ simply by $X$ when there is no risk of confusion.

\begin{definition}
	A \dtext{torified morphism} $\Phi:(X,T) \lto (Y,S)$ between torified schemes $X$ and $Y$ with torifications $T=\{T_i \overset{\tau_i}{\hookrightarrow} X\}_{i\in I}$ and $S=\{S_j \overset{\sigma_j}{\hookrightarrow} Y\}_{j\in J}$, respectively, is a triple $\Phi=(\vphi,\tilde{\vphi},\{\vphi_i\}_{i\in I})$ where
	\begin{itemize}
		\item $\vphi:X\to Y$ is a morphism of schemes,
		\item $\tilde{\vphi}:I\to J$ is a set map, and
		\item $\vphi_i:T_{i}\to S_{\tilde{\vphi}(i)}$ are morphisms of algebraic groups such that for all $i\in I$ the diagram
			\[
				\xymatrix{
					X \ar[rr]^{\vphi} & & Y \\
					T_i \ar[u]^{\tau_i} \ar[rr]^{\vphi_i} && S_{\tilde{\vphi}(i)} \ar[u]_{\sigma_{\tilde{\vphi}(i)}}
				}
			\]
		commutes.
	\end{itemize}
\end{definition}

The \emph{category of torified schemes} consists of torified schemes together with torified morphisms. The \emph{category of torified varieties} is defined as the full subcategory of the category of torified schemes. The \emph{category of affinely torified varieties} is defined as the full subcategory of affinely torified varieties. 

The following lemma shows that torified morphisms between affinely torified schemes is determined by its restriction to affine open torified subschemes. We can subsume this fact by saying that a torified morphism between affinely torified schemes is \emph{affinely torified}.

\begin{lemma}\label{lemma: affinely torified morphisms}
 Let $\Phi: (X,T) \to (Y,S)$ be a torified morphism between affinely torified schemes. Then there is an affine open cover $\{U_i\}_{i\in I}$ of $X$ that respects the torification $T\to X$ and an affine open cover $\{V_i\}_{i\in I}$ that respects the torification $S\to Y$ such that $\vphi:X\to Y$ restricts to a morphism $\vphi(i):U_i\to V_i$ for all $i\in I$.
\end{lemma}

\begin{proof}
 Since $X$ and $Y$ are affinely torified, we can choose an affine open cover $\{\tilde U_j\}_{j\in J}$ of $X$ that respects the torification $T\to X$ and an affine open cover $\{\tilde V_k\}_{k\in K}$ of $Y$ that respects the torification $S\to Y$. Define the intersection $U_i=\tilde U_j\times_X \vphi^{-1}(\tilde V_k)$ and $V_i=\tilde V_k$ for every pair $i=(j,k)\in J\times K= I$. Then the following is clear: $\vphi:X\to Y$ restricts to a morphism $\vphi(i):U_i\to V_i$ for all $i\in I$; the collection $\{U_i\}_{i\in I}$ is an open cover of $X$ that respects the torification $T\to X$ and $\{V_i\}_{i\in I}$ is an open affine cover of $Y$ that respects the torification $S\to Y$.

 The lemma is proven if we can show that the subschemes $U_i$ are affine for all $i\in I$. First note that for every $i\in I$, the scheme $U_i$ is quasi-affine since it embeds into $\tilde U_j$ where $i=(j,k)$. Therefore $U_i$ embeds into the affine subscheme $Z_i=\Spec\O_X(U_i)$ of $\tilde U_j$. The morphism $\vphi(i):U_i\to V_i$ extends to a morphism $\vphi(i)':Z_i\to V_i$ since $V_i$ is affine. This means that $Z_i$ is contained in $\vphi^{-1}(\tilde V_k)$. Therefore $Z_i$ is contained in $\tilde U_j\times_X\vphi^{-1}(\tilde V_k)=U_i$, which shows that $U_i=Z_i$ is affine.
\end{proof}

	\begin{lemma}\label{lemma:productoftorified}
		Let $X$, $Y$ be (affinely) torified schemes over $\Z$, then the cartesian product $X\times Y$ is also (affinely) torified.
	\end{lemma}
	\begin{proof}
		If $X=\decomp_{i\in I} T_i$ and $Y=\decomp_{j\in J} S_j$ are (affine) torifications of $X$ and $Y$, then we have that $X\times Y=		\decomp_{(i,j)\in I\times J} T_i\times S_j$ is an (affine) torification of $X \times Y$.
	\end{proof}

	\begin{lemma}\label{lemma:decompoftorified}
		If $X=\decomp_{i\in I} X_i$ is a decomposition of $X$ into torified schemes $X_i$, then $X$ is also torified.
	\end{lemma}
	\begin{proof}
		If for each $X_i$ we have $X_i= \decomp_{j\in J_i}T_j$, then $\decomp_{j\in \bdu_{i\in I}J_i} T_j$ is a torification of $X$.
	\end{proof}

	
\subsection{Zeta functions over $\F_1$}

 One expects a certain zeta function $\zeta_X$ of a geometric object $X$ over $\F_1$
 that actually does not depend on the particular geometry,
 but is the ``limit $q$ goes to $1$'' of the zeta functions of the base extensions $X_{\F_q}=X\otimes_\Fun\F_q$.
 We recall the precise notion of a zeta function over $\F_1$ and calculate it in the case that $X\otimes_\Fun\Z$ is a torified variety.

 Assume that there is a polynomial $N(T)\in\Z[T]$ such that
 $N(q)=\#X_{\F_q}(\F_q)$ whenever $q$ is a prime power.
 This polynomial is called the \dtext{counting function of $X$}.
 Using the formal power series
 	\[
		Z(q,T):=\exp\left( \sum_{r\geq 1}N(q^r)T^r/r \right)\;, 
	\]
 we define the \dtext{zeta function of $X$} as
 	\[
		\zeta_X(s):= \lim_{q\to 1} Z(q,q^{-s})(q-1)^{N(1)}\;.
	\]
 We have the following result.

\begin{theorem}[Soul\'{e}]\label{thm:zetaformula}
	The function $\zeta_X(s)$ is a rational function with integral coefficients. Moreover, if $N(x)=a_0 + a_1x +\dotsb + a_dx^d$, then we have
		\begin{equation*}
			\zeta_X(s) = \prod_{i=0}^d (s-i)^{-a_i}.
		\end{equation*} 	
\end{theorem}

\begin{proposition}\label{prop:counting}
 Let $X=\decomp T_i$ be a torified variety. Put $I^{(l)}:=\{i\in I|\ \dim T_i= l\}$ and $\delta_l:=\#I^{(l)}$.
 Then $X$ has a counting function, which is given by 
			\begin{equation*}\label{eq:counting}
				N(q)=\sum_{l=0}^{\dim X} \delta_l (q-1)^l \in \Z[q]\;.
			\end{equation*}
 In particular, the numbers $\delta_l$ are independent from the chosen torification of $X$.			
\end{proposition}
	
	\begin{proof}
      The form of the counting function follows from $\#\Gm^l(\F_q)=(q-1)^l$ and from Lemma \ref{lemma:decomp}. The independence of the $\delta_l$ from the torification can be seen as follows. Let $T$ and $S$ be two torifications of $X$, and denote by $N_T(q)$ and $N_S(q)$ the corresponding counting functions. For every finite field $\F_q$ we have
			\[
				N_{T}(q) = \card{\bdu T_i(\F_q)} = \card{X(\F_q)} = \card{\bdu S_j(\F_q)} = N_S(q),
			\]
		so $N_T(q)$ and $N_S(q)$ coincide in an infinite number of values, and henceforth they must be equal as polynomials.
	\end{proof}

 With this, we can calculate the zeta function for a model $X$ of a torified variety over $\F_1$.
 Let the numbers $\delta_l$ be defined as in the proposition. Then
	\begin{equation*}\label{eq:counting2}
		N(q) \quad = \quad \sum_{l=0}^{\dim X} \delta_l\, (q-1)^l 
		\quad = \quad \sum_{l=0}^{\dim X} \ \Biggl(\sum_{k=l}^{\dim X} (-1)^{k-l}\, \binom{k}{l}\,\delta_k \Biggr)\ q^l\;,
	\end{equation*}
 from where we can compute the zeta function of a torified variety by applying Theorem \ref{thm:zetaformula}.
 It is possible to recover all examples of zeta functions in \cite{Kurokawa2005} by this method 
 since all these examples concern torified varieties as explained in the following example section.


\subsection{Examples of torified varieties}
\label{section:torified_examples}

\subsubsection{Tori and the multiplicative group}
\label{subsubsection:Gm}

If $X=\Gm^d$ is a product of multiplicative groups, it admits the obvious torification given by the identity map $\Gm^d\to X$.

\subsubsection{The affine spaces $\A^n$.}
\label{subsubsection:A^n}

The affine line admits a torification $\A^1= \Gm^0 \du \Gm^1$, obtained by choosing any point as the image of $\Gm^0$ and identifying its complement with $\Gm^1$.

By applying Lemma \ref{lemma:productoftorified}, and taking into account that $\Gm^r\times \Gm^s\cong \Gm^{r+s}$ we obtain a torification of the affine spaces by
	\begin{equation*}
		\A^n = \Gm^0 \du n \Gm^1 \du \dotsb \du \binom{n}{d} \Gm^d \du \dotsb \du \Gm^n,
	\end{equation*}
	where by $r\Gm^d$ we mean that we get $r$ different copies of the torus $\Gm^d$.


 \subsubsection{Toric varieties}
 \label{section_toric_as_torified}
 
 As a general reference for toric varieties consider \cite{Fulton1993} or \cite{Maillot2000}.
 We introduce the notation for toric varieties that is frequently used in this paper.
 Let $\Delta$ be a fan, i.e. a family of pairwise distinct cones ordered by inclusion such that the faces of a cone in $\Delta$ are in $\Delta$
 and such that the intersection of two cones in $\Delta$ is a face of each of the cones
 (cones are always assumed to be embedded in $\R^n$ and to be strictly convex and rational).
 To a cone $\tau\subset\R^n$ of $\Delta$, we associate the semi-group $A_\tau=\tau^\vee\cap(\Z^n)^\vee$, 
 where $\tau^\vee\subset(\R^n)^\vee$ is the dual cone of $\tau$ and $(\Z^n)^\vee$ is the dual lattice to $\Z^n$ in the standard basis of $\R^n$.
 We put $U_\tau=\Spec\Z[A_\tau]$.
 An inclusion $\tau\subset\tau'$ defines an inclusion of semi-groups $A_{\tau'}\subset A_\tau$ 
 and an open immersion of schemes $U_\tau\hookrightarrow U_{\tau'}$.
 Then the toric variety $X$ associated to $\Delta$ is the direct limit of the family $\{U_\tau\}_{\tau\in\Delta}$ 
 relative to the immersions $U_\tau\hookrightarrow U_{\tau'}$.
 In the following, we will always consider toric varieties $X$ together with a fixed fan $\Delta$.

 A morphism $\Delta\to \Delta'$ of fans of toric varieties $X$ and $X'$, respectively, is map $\tilde\psi$ between ordered sets 
 together with a direct system of semi-group morphisms $\psi_\tau:\tau\to\tilde\psi(\tau)$ (with respect to inclusion of cones) 
 whose dual morphisms restrict to $\psi_\tau^\vee: A_{\tilde\psi(\tau)}\to A_\tau$, where $\tau$ ranges through $\Delta$.
 Taking the direct limit over the system of scheme morphisms $\Spec\Z[\psi_\tau^\vee]:U_\tau\to U_{\tilde\psi(\tau)}$
 yields a morphism $\psi:X\to X'$ between toric varieties.
 A triple $(\psi,\tilde\psi,\{\psi_\tau\})$ like this is called a toric morphism.
 The \emph{category of toric varieties} consists of toric varieties over $\Z$ together with toric morphisms.
 
Let $X$ be a toric variety with fan $\Delta$. 
Let $A_\tau^\times$ be the group of invertible elements of $A_\tau$, then the algebra morphism
	\begin{eqnarray*}
		\Z[A_\tau] & \lto & \Z[A_\tau^\times] \\
		a & \lmto & 	\left\{
						\begin{array}{ll}
							a & \text{if $a\in A_\tau^\times$},\\
							0 & \text{if $a\in A_\tau\setminus A_\tau^\times$}
						\end{array}
					\right.
	\end{eqnarray*}
	defines an immersion of the torus $T_\tau = \Spec\Z[A_\tau^\times]$ into $U_\tau\subseteq X$, and we obtain the well-known decomposition of $X$ into tori $T_\tau$ (cf. \cite[\S4, Prop. 2]{Demazure1970}, \cite[\S 3.1]{Fulton1993} or \cite[Proposition 2,2,14]{Maillot2000}), that in our formulation reads as follows:

	\begin{proposition}
		The family $T_{\Delta}=\{T_\tau\hookrightarrow X\}_{\tau\in \Delta}$ is a torification of $X$.
	\end{proposition}

Given a toric morphism $(\psi, \tilde{\psi}, \{\psi_\tau\}):(X,\Delta)\to (X',\Delta')$, we obtain a torified morphism
$(\vphi,\tilde{\vphi}, \{\vphi_\tau \}):(X,T_\Delta) \lto (X', T_{\Delta'})$ as follows:
	\begin{itemize}
		\item $\vphi = \psi:X \to X'$,
		\item $\tilde{\vphi}= \tilde{\psi}:\Delta \to \Delta'$,
		\item since the map $\psi_\tau: A_{\tilde{\psi}(\tau)} \to A_\tau$ preserves units, it restricts to a map $A_{\tilde{\psi}(\tau)}^\times \to A_\tau^\times$, and therefore it induces a homomorphism of tori $\vphi_\tau := \Spec \Z[\psi_\tau] : T_\tau \to T_{\tilde{\psi}(\tau)}= T_{\tilde{\vphi}(\tau)}$.
	\end{itemize}

\begin{remark}
	The triple $(\vphi,\tilde{\vphi}, \{\vphi_\tau \}):(X,T_\Delta) \lto (X', T_{\Delta'})$ is indeed a torified morphism: the diagram
		\[
			\xymatrix{
				X \ar[r]^{\vphi} & X' \\
				T_\tau \ar[u] \ar[r]_{\vphi_\tau} & T_{\tilde{\vphi}(\tau)} \ar[u]
			}
		\]
		commutes because
		\[
			\xymatrix{
				X \ar[rr]^{\psi} && X' \\
				U_\tau \ar[u] \ar[rr]_{\Spec\Z[\psi_\tau^\vee]} && U_{\tilde{\vphi}(\tau)} \ar[u]
			}
		\]
		does.
\end{remark}

Since $\{U_\tau\}$ is an affine open cover that is compatible with the torification that we have constructed and since every toric morphism is covered by morphisms between affine open toric subvarieties, we have the following result.

\begin{proposition}
	The torifications associated to toric varieties are affine.
\end{proposition}


\subsubsection{Grassmannians and their Schubert varieties}
 \label{section_Grassmannians}
	For a couple of positive integers $0\leq k\leq n$, the \dtext{Grassmann variety} $\Gr(k,n)=\Gr_k(\mathbb{A}^n)$ is defined as the variety of $k$--planes in the affine space $\mathbb{A}^n$ (cf. \cite[Chapter 14]{Fulton1998}). 

	 The Grassmann varieties admit a nice decomposition in \dtext{Schubert cells} (cf. \cite[Chapter 1.\S 5]{Griffiths1978} and \cite[Chapter 14, \S 6]{Fulton1998}) indexed by the set of multi-indices 
	\[
		I_{k,n} := \{\underline{i}=(i_1,i_2,\dotsc,i_k)|\ 1\leq i_1 < i_2 < \dotsb < i_k \leq n \},
	\]
	partially ordered by $(i_1,i_2,\dotsc,i_k)\leq(j_1,j_2,\dotsc,j_k)$ if and only if $i_l\leq j_l$ for $l=1,\dotsc,k$. To each element $\underline{i}$ of $I_{k,n}$ we can associate the \dtext{Schubert variety} $X_{\underline{i}}$ and the \dtext{Schubert cell} $C_{\underline{i}}$. The Schubert varieties give a stratification of the Grassmannian, with $X_{\underline{i}}\subseteq X_{\underline{j}}$ if and only if $\underline{i}\leq \underline{j}$, we have $\Gr(k,n)=X_{\underline{i_m}}$, where $\underline{i_m}=(n-k+1,\dotsc,n)$. Moreover, we have the following result (see \cite[Chapter 1.\S 5]{Griffiths1978} for details):
	
	\begin{theorem}[Schubert decomposition]
		Each Schubert cell $C_{\underline{i}}$ is an affine space of dimension $\dim C_{\underline{i}} = \sum_{t=1}^k (i_t-t)$, and we have the cell decomposition
			\[
				X_{\underline{j}}=\decomp_{\underline{i}\leq \underline{j}} C_{\underline{i}}.
			\]
	\end{theorem}

	As an immediate consequence, applying Lemma \ref{lemma:decompoftorified} and the previous example, we obtain a torification for all Schubert varieties, and in particular for the Grassmann varieties.
 
 \begin{example}\label{ex:gr24}
 Let us illustrate this example in the particular case of the Grassmannian $\Gr(2,4)$. This example is of particular interest in connection with the open problem of realizing $\Gr(2,4)$ as a variety over $\F_1$, which was posed by Soul\'e in \cite[section 5.4]{Soule2004}. For the set $I_{2,4}$ we get, with its partial ordering
  \[
   I_{2,4} =
    \begin{array}{l}
      \xymatrix@R=0.2pc{
       & & (1,4) \ar[dr] \\
      (1,2) \ar[r] & (1,3) \ar[ur] \ar[dr] & & (2,4) \ar[r] & (3,4)\\
      & & (2,3) \ar[ur]
     }
    \end{array}
  \]
  generating the corresponding Schubert cells
  \[
   C_{1,2}\cong \A^0,\ C_{1,3}\cong \A^1,\ C_{1,4}\cong C_{2,3}\cong \A^2,\ C_{2,4}\cong \A^3,\ C_{3,4}\cong \A^4,
  \]
  that lead to the torification
  \begin{eqnarray*}
   \Gr(2,4) & = & C_{1,2} \du C_{1,3}\du C_{1,4}\du C_{2,3}\du C_{2,4}\du C_{3,4} = \\
   & = & \A^0 \du \A^1 \du 2 \A^2 \du \A^3 \du \A^4 = \\
   & = & 6\Gm^0 \du 12\Gm^1 \du 11\Gm^2 \du 5\Gm^3 \du \Gm^4.
  \end{eqnarray*}
   
 It is worth noting that the above torification is \emph{not} compatible with the usual affine open cover of $\Gr(2,4)$ by six $4$-dimensional affine spaces, which comes from embedding $\Gr(2,4)$ into $\mathbb P^5$ via the Pl\"ucker map and intersecting the image with the canonical atlas of $\mathbb P^5$. Namely, all $6$ opens are needed to cover $\Gr(2,4)$, but the intersection of all opens does not contain a $4$-dimensional torus as a subvariety. This shows that in general we cannot expect the Grassmann varieties to be affinely torified.
 \end{example}


\subsubsection{Flag varieties}
 \label{section_flag_varieties}
Let $V$ be a linear bundle (over a point) of rank $n$. For each $m$--tuple $(d_1,\dotsc,d_m)$ of positive integers, with $d_1+ \dotsb + d_m = n$, a \dtext{flag} of type $(d_1,\dotsc,d_m)$ consists of an increasing sequence of linear sub-bundles
	\[
		0 = V_0 \subset V_1 \subset V_2 \subset \dotsb \subset V_m
	\]
such that $\rk (V_j/V_{j-1})=d_j$ for all $j=1,\dotsc, m$. The set $X(d_1,\dotsc, d_m)$ of all flags of type $(d_1,\dotsc, d_m)$ is a scheme, known as the \dtext{flag variety} of type $(d_1,\dotsc, d_m)$. For instance, the flag variety $X(d,n-d)$ coincides with the Grassmannian $\Gr(d,n)$.

As in the case of the Grassmann varieties, flag varieties admit a decomposition in Schubert cells, though their description is in general more complicated. The underlying idea to this approach is the realization of the flag variety $X(d_1,\dotsc, d_m)$ as the quotient $GL_n/ P(d_1,\dotsc, d_m)$, where $P = P(d_1,\dotsc, d_m)$ is the standard parabolic subgroup of $GL_n$ consisting of block upper-triangular invertible matrices with blocks of sizes $d_1,\dotsc, d_m$. The Schubert cells and varieties are then parametrized by the coset space $S_n/(S_{d_1}\times S_{d_2}\times \dotsb \times S_{d_m}) = W/W_P$. Each right coset modulo $W_P$ contains a unique representative $w$ such that we have 
	\begin{gather*}
		w(1)<w(2)<\dotsb < w(d_1), \\
		w(d_1+1)<w(d_1+2)<\dotsb < w(d_1+d_2) \\
		\vdots \\
		w(d_1 + \dotsb + d_{m-1} + 1 ) < \dotsb < w(d_1+\dotsb + d_m).
	\end{gather*}
This defines the set $W^P$ of minimal representatives of $W/W_P$. The Schubert cells in $GL_n/P$ are the orbits $C_{wP}:=(BwP)/P$, where $B$ denotes again the Borel group consisting of all the upper triangular matrices, and the Schubert varieties $X_{wP}$ are defined as the closures of the Schubert cells. A detailed description of this decomposition can be found in \cite[Section 10.2]{Fulton1997}.
 
	Exactly as it happened with the Grassmannians, Lemma \ref{lemma:decompoftorified} applied to the Schubert cell decomposition provides a torification of the flag varieties and their Schubert subvarieties.  

	\begin{example}[Complete flag varieties]
		Consider the flag variety $X=X(1,\dotsc, 1)$, that can be identified with the quotient $GL_n/B$. In this case, $P(1,\dotsc, 1)=B$ the group of upper-triangular matrices, and we have $W_P=\{e\}$ the trivial group, and thus Schubert cells are parametrized by elements of the Weyl group $W=S_n$. Associated to each permutation $w\in S_n$ we construct the complete flag
			\[
				F_w:= 0\subset \langle e_{w(1)} \rangle \subset \dotsb \subset \langle e_{w(1)},\dotsc, e_{w(k)} \rangle \subset \dotsb.
			\]
		Schubert cells are given by $C_w = BF_w$, we can explicitly compute the dimension as $\dim C_w = l(w)$, the length of the permutation $w$, and we have the decomposition
			\[
				X(1,\dotsc, 1) = \decomp_{w\in S_n} \A^{l(w)},
			\]
		that we can turn into a torification in the same way we did for the Grassmannian $\Gr(2,4)$. As it happened for the Grassmann varieties, in general it is not clear whether the above torification is affine.
	\end{example}


\subsubsection{Chevalley schemes}
 \label{section_Chevalley}
 We establish an affine torification for split Chevalley schemes.
 As general reference, see \cite[Expose XXI and XXII]{SGA3} or the survey in \cite[Section 4]{Connes2008}.
 
 Let $G$ be a split Chevalley scheme over $\Z$ with maximal split torus $T$. 
 Let $N$ be the normalizer of $T$ in $G$ and $W=N(\C)/T(\C)$ be the Weyl group.
 Let $B$ be a Borel subgroup of $G$ that contains $T$ and has unipotent radical $U$.
 Let $\Phi$ be the set of roots and let $\Phi^+\subset\Phi$ be the set of positive roots corresponding to $B$.
 Let $X_r$ denote the additive $1$-parameter subgroup of $G$ defined by $r\in\Phi$.
 Put $\Phi_w=\{r\in\Phi^+\mid w(r)<0\}$ and let $U_w$ be the subgroup of $U$ that is generated by $\{X_r\}_{r\in\Phi_w}$.
 Choose a set of representatives $\{n_w\}_{w\in W}$ for $W$ in $N(\Z)$.
 We restate the Bruhat decomposition of $G$ in the language of the present paper.
 
\begin{theorem}[Bruhat decomposition]\label{Bruhat_decomposition}\ \\ 
 The family of inclusions of subschemes $\{U_wn_wTU\hookrightarrow G\}_{w\in W}$ is a decomposition of $G$.
\end{theorem}
 
 We refer to SGA3 (\cite[Expose XXII, Thm.\ 5.7.4 and Rem.\ 5.7.5]{SGA3}) for a proof.
 
\begin{proposition}
 \label{thm_torified_Chevalley_schemes}
 Let $G$ be a split Chevalley scheme. 
 Then there exists an affine torification $S$ of $G$.
\end{proposition} 
 
\begin{proof} 
 Let $r$ be the dimension of $T$, let $s$ be the dimension of $U$ and for every $w\in W$, 
 let $s_w$ be the dimension of $U_w$, which equals the cardinality of $\Phi_w$.
 Then, as a scheme, $U_wn_wTU$ is isomorphic to $\A^{s_w}\times\mathbb G_m^r\times\A^s$ for every $w\in W$.
 Since affine space and the multiplicative group scheme are torified, 
 Lemma \ref{lemma:productoftorified} implies that $U_wn_wTU$ is torified,
 and Theorem \ref{Bruhat_decomposition} together with Lemma \ref{lemma:decompoftorified} implies that $G$ is torified.
 Since $G$ is an affine scheme, $G$ is affinely torified.
\end{proof} 

\begin{example}
 \label{example:Sl(2)}
 Let $G=\Sl(2)$. 
 Let $T$ be the diagonal torus, $N$ its normalizer in $G$ and $B$ the subgroup of upper triangular matrices.
 Let $e\ = \ \bigl( \begin{smallmatrix} 1 & 0 \\ 0 & 1 \end{smallmatrix} \bigr)$ and 
 $w\ = \ \bigl( \begin{smallmatrix} 0 & 1 \\ -1 & 0 \end{smallmatrix} \bigr)$,
 then $\{e,w\}\subset N(\Z)$ represents the Weyl group $W$. 
 In the notation of the proof of Theorem \ref{thm_torified_Chevalley_schemes},
 we have $r=s=s_e=1$ and $s_w=0$, and thus we have decompositions
 $$ N \ = \ \Gm \ \du \ \Gm \quad  \subset \quad     G \ = \ \Gm\times\A^2 \ \du \ \Gm\times\A \ = \ 2\Gm \ \du \ 3\Gm^2 \ \du \ \Gm^3 \;. $$
\end{example} 

\begin{remark}
 As established in Proposition \ref{prop:counting}, the counting function $N(q)$ for every torifiable variety is a polynomial with non-negative integral coefficients in $q-1$. Thus, any variety which counting function does not satisfy this condition cannot be torified. For instance, the variety $\mathbb{P}^1\setminus \{0,1,\infty\}$, the projective line minus three points, has counting function $N(q) = q-2 = (q-1)-1$ and thus is not torifiable. 

 Another property of an irreducible torified variety is that it is rational because it contains a dense open subscheme that is isomorphic to a split torus and split tori themselve are rational varieties. For instance, elliptic curves cannot be torified.

 It is worth noting that there are examples of rational varieties whose counting function is a polynomial in $q-1$ with nonnegative integral coefficients that cannot be torified. For instance, one might consider the complex cone $K(\C)=\{ (x,y,z)\mid z^2=xy\}$ in $\A^3(\C)$, which is already defined over $\Z$. It has counting function $N_K(q)=q^2$. The complement $Y(\C)=\A^3(\C)\setminus K(\C)$ is also defined over $\Z$ and has counting function $N_Y(q)=q^3-q^2=(q-1)^3+2(q-1)^2+(q-1)$. Since it is dense and open in $\A^3\subset\mathbb{P}^3$, it is a rational variety. If $Y$ was torifiable, it would yield an embedding $\Gm^3(\C)\hookrightarrow Y(\C)$, which extends to an automorphism of $\mathbb{P}^3(\C)$, and the inverse map would establish an embedding of $K$ into $\mathbb{P}^2$, which does not exist. 
\end{remark}


\section{Connes and Consani's geometry}
 \label{section_CC-gadgets}

\subsection{CC--gadgets and CC--varieties}
Let us start this section by recalling some definitions from \cite{Connes2008}. 

	\begin{definition}
		A \dtext{(Connes-Consani) gadget} over $\Fun$ (CC--gadget for short) is a triple $X=(\underline{X}, X_\C, \ev_X)$ where
			\begin{itemize}
				\item $\underline{X}:\mathcal{F}_{ab}\to \Sets$ is a functor from the category of finite abelian groups to the category of sets,
				\item $X_\C$ is a variety over $\C$ and
				\item $\ev_X:\underline{X}\Longrightarrow \Hom(\Spec\C[-],X_\C)=X_\C(\C[-])$ is a natural transformation.
			\end{itemize}
		We say that a gadget $X$ is \dtext{finite} if $\underline{X}(D)$ is finite for all finite abelian groups $D$, and that it is \dtext{graded} if $\underline{X}=\bdu_{l\geq 0}\underline{X}^{(l)}$ is a graded functor.	
			
		A \dtext{morphism of CC--gadgets} $\vphi:(\underline{X},X_\C,\ev_X)\lto (\underline{Y},Y_\C,\ev_Y)$ consists of a pair $(\underline{\vphi}, \vphi_\C)$ where
		\begin{itemize}
			\item $\underline{\vphi}:\underline{X}\Longrightarrow \underline{Y}$ is a natural transformation and
			\item $\vphi_\C:X_\C\lto Y_\C$ is a morphism over $\C$
		\end{itemize}
		such that for all finite abelian groups $D$ the diagram
		\[
			\xymatrix{
				\underline{X}(D) \ar[rr]^{\underline{\vphi}(D)} \ar[d]_{\ev_X(D)} && \underline{Y}(D) \ar[d]^{\ev_Y(D)}\\
				X_\C(\C[D]) \ar[rr]_{\vphi_\C(\C[D])} && Y_\C(\C[D])
			}
		\]
		commutes.
		
		A morphism of gadgets $\vphi$ is an \dtext{immersion} if for every finite abelian group $D$ the map $\underline{\vphi}(D)$ is injective, and $\vphi_\C$ is an immersion of schemes.
	\end{definition}

	We will say that a CC--gadget is affine, projective, irreducible, et-cetera, if $X_\C$ is so.

	\begin{definition}\label{definition_CC-variety}
		Given a reduced scheme $X$ of finite type over $\Z$, we define the \dtext{CC--gadget $\mathcal{G}(X)$ associated to $X$} by $\mathcal{G}(X):=(\underline{X},X_\C, \ev_X)$, where
			\begin{itemize}
				\item $\underline{X}(D):=\Hom(\Spec \Z[D], X)=X(\Z[D])$ for every $D$,
				\item $X_\C:=X\otz \C$ and
				\item $\ev_X:X(\Z[-])\Longrightarrow X_\C(\C[-])$ is given by extension of scalars. 
			\end{itemize}
		A morphism of schemes $\vphi:X\to Y$ induces a morphism of CC--gadgets $\mathcal{G}(\vphi):\mathcal{G}(X)\to \mathcal{G}(Y)$ defined by $\mathcal{G}(\vphi):=(\underline{\vphi},\vphi_\C)$, where
			\begin{itemize}
				\item $\underline{\vphi}=\vphi^\ast$ is the pullback by $\vphi$, i.e. for all $f:\Spec \Z[D]\to X$ we set $\underline{\vphi}(f):=\vphi\circ f$.
				\item $\vphi_\C:=\vphi \otz \C$ is the complexification of $\vphi$.
			\end{itemize}
  A finite graded CC--gadget $X=(\underline{X}, X_\C, \ev_X)$ is an \dtext{affine variety over $\Fun$} in the sense of Connes-Consani if there is a reduced affine scheme $X_\Z$ of finite type over $\Z$ and an immersion $i:X\to \mathcal{G}(X_\Z)$ such that for all affine reduced schemes $V$ of finite type over $\Z$ and all morphisms of CC--gadgets $\psi:X\to \mathcal{G}(V)$, there is a unique morphism $\vphi:X_\Z\to V$ of schemes such that the diagram
   \[
    \xymatrix{
     X \ar[r]^i \ar[dr]_{\psi} & \mathcal{G}(X_\Z) \ar@{-->}[d]^{\mathcal{G}(\vphi)}\\
      & \mathcal{G}(V)
    }
   \]
  commutes. If $X=(\underline{X},X_\C, \ev_X)$ is an affine variety over $\Fun$, we say that $X_\Z$ is the \dtext{extension of scalars of $X$ to $\Z$}, and we write $X_\Z =: X\otf \Z$. By Yoneda's lemma, $X_\Z$ is unique up to unique isomorphism.
\end{definition}

Note that we have substituted ``variety over $\Z$'' of the original definition in \cite{Connes2008} by ``reduced scheme of finite type over $\Z$''.
It is however not an issue to abandon the restraint of reducibility (in accordance with Soul\'e's convention), 
also cf.\ Remark \ref{remark:S_to_CC_problems}.

 If we have a morphism of CC--gadgets $\vphi=(\underline{\vphi},\vphi_\C):X\to Y$, and $X$ and $Y$ are affine varieties over $\Fun$, then the universal property of $Y$ yields an immersion $i_Y:Y\to \mathcal{G}(Y_\Z)$. Hence, we get a morphism $i_Y\circ \vphi:X\to \mathcal{G}(Y_\Z)$. By the universal property of $X$, we obtain a morphism $\vphi_\Z:X_\Z\to Y_\Z$ of schemes. We will write $\vphi_\Z=:\vphi\otf \Z$, and say that $\vphi_\Z$ is the \dtext{extension of scalars of $\vphi$ to $\Z$}.

 We shall restrict in this work to the class of varieties over $\F_1$ whose functor represents the counting function of the base extension to $\Z$, as explained below.
 
\begin{definition}
 An affine variety $X=(\uline X, X_\C,\ev_X)$ over $\F_1$ is called an \dtext{affine CC-variety} if for every prime power $q$ and every abelian group $D$ of cardinality $q-1$, we have $\#\uline X(D)=\# X_\Z(\F_q)$; when this latest property is satisfied, we say that the functor $\uline{X}$ \emph{represents the counting function} of $X_\Z$, by what we simply mean that it counts the right number of points.
\end{definition}

 We transfer Connes and Consani's definition of a  variety over $\F_1$ (cf.\ \cite[para.\ 3.4]{Connes2008} for an explanation on how to go from the affine to the general case) to this restricted class. This yields a class of $\Fun$--varieties whose functor represents its counting function (in the same sense as for affine $\Fun$-vaieties). It was suggested by Connes and Consani themselve (\cite[Section 3]{Connes2008}) that this would be a meaningul restriction. We need this restriction to construct in section \ref{subsection:CC_to_S} the functor $\FCCS$ from Connes-Consani's $\Fun$-geometry to Soul\'e's $\Fun$-geometry. We will see, however, that this restricted class contains all examples from \cite{Connes2008}.

\begin{definition}
 Let $X=(\underline{X},X_\C,\ev_X)$ and $U=(\underline{U},U_\C,\ev_U)$ be finite graded CC--gadgets over $\Fun$.
 A \dtext{graded morphism} is a morphism $(\uline\vphi,\vphi_\C):U\to X$ such that $\uline\vphi(D)$ restricts to
 a map $\uline\vphi^{(l)}(D):\uline U^{(l)}(D)\to \uline X^{(l)}(D)$ for every $l\geq 0$ and every finite abelian group $D$.
 A morphism $(\uline\vphi,\vphi_\C):U\to X$ of finite graded CC--gadgets is called an \dtext{open immersion}
 if it is a graded immersion such that $\vphi_\C: U_\C\to X_\C$ is an open immersion and if
 $$ \uline\vphi(\underline{U}(D)) \ = \ \{x\in \underline{X}(D)|\ \im(\ev_X(x))\subseteq U_\C \}\;. $$
 If such an open immersion is fixed, $U$ is called an \dtext{open CC--subgadget} of $X$.
 
 An open affine cover of $X$ is a family $\{U_i\}_{i\in I}$ of open affine CC--subgadgets such that 
 $\bigcup_{i\in I}\uline{U_i}(D)=\uline X(D)$ and $\{U_{i,\C}\}$ is an open affine cover of $X_\C$.
 A \dtext{CC--variety} is a finite graded CC--gadget $X$ that has an open affine cover by affine CC--varieties.

 If $U$ is an open CC--subgadget of a CC--variety $X$ that is a CC--variety itself, we call $U$ an \dtext{open CC--subvariety} of $X$.
 Let $\{U_j\}$ be the family of all open CC--subvarieties of $X$.
 The \dtext{extension of scalars} (or the \dtext{base extension}) of $X$ from $\F_1$ to $\Z$ is the
 direct limit over the family $\{U_{j,\Z}\}$ relative to all canonical inclusions,
 and it is denoted by $X_\Z=X\otimes_\Fun\Z$.

 \label{def_morphism_of_CC-varieties}
 A \dtext{morphism of CC--varieties} $\varphi:X\to Y$ is a gadget morphism between CC--varieties $X$ and $Y$ such that the family $\{X_j\}_{j\in J}$ of all open affine CC--subvarieties whose images $Y_j$ under $\varphi$ is affine covers $X$.
\end{definition}

Note that in section \ref{subsection:CC_to_S}, we will restrict the morphisms between CC-varieties in order to compare Connes-Consani's notion with the one of Soul\'e. 

Further note that the intersection of two open CC--subvarieties of a given CC--variety $X$ is again a CC--subvariety.
 This implies that the functor of $X$ represents the counting function of $X_\Z$.
 Further, we have a base extension to $\Z$ for morphisms between CC--varieties. 
 More precisely, Lemma \ref{lemma_univ_prop} holds, mutatis mutandis, for CC--varieties.


\subsection{Affinely torified varieties as CC--varieties}
\label{subsection:torified_CC-varieties}
Let $(X,T)$ be a torified variety. We define a CC--gadget $\mathcal{L}(X,T):=(\underline{X},X_\C, \ev_X)$ over $\Fun$ consisting of the following data:
	\begin{itemize}
		\item The graded functor $\underline{X}= \{\underline{X}^{(l)}\}_{l\geq 0}$ defined by
			\begin{eqnarray*}
				\underline{X}^{(l)} :\mathcal{F}_{ab} & \lto & \Sets \\
				D & \lmto & \bdu_{i\in I^{(l)}} \Hom(A_i, D)
			\end{eqnarray*}
		for every $l\geq0$, where $I^{(l)}=\{i\in I\mid \dim T_i=l\}$ and $A_i:=\Hom_{alg-gr}(T_i,\Gm)$.
		\item The complex variety $X_\C:=X\otz \C$.
 		\item For every $i\in I$, the evaluation 
			\[
				\ev_X(D): \Hom(A_i,D) \hookrightarrow \Hom(\C[A_i],\C[D])\subset \Hom(\Spec \C[D],X_\C).
			\]
	\end{itemize}

The following is a well known result in the theory of algebraic groups (cf. \cite[section 1.5]{Borel1991} or \cite[section 1.4]{Waterhouse}).

\begin{proposition}
	Let $G=\Spec R$ and $T=\Spec S$ be affine algebraic groups. The coordinate rings $R$ and $S$ are Hopf algebras, and we have the equality
		\[
 			\Hom_{alg-gr}(G,T)=\Hom_{Hopf}(S,R).
		\]
\end{proposition} 
	
	In our particular situation if $R=\Z[D]$ and $S=\Z[A]$ are group rings for some abelian groups $A$ and $D$, with Hopf algebra structure given in the usual way, since group-like elements in the Hopf algebra $\Z[A]$ are precisely the elements of $A$, we have 
		\[
			\Hom_{Hopf}(\Z[A],\Z[D])=\Hom(A,D),
		\]
	and we obtain the following consequence:

\begin{corollary}
	Let $A$ be a free abelian group of rank $d$, $\Z[A]$ its group ring with the usual Hopf algebra structure, and $T=\Spec \Z[A]$ the torus of $A$. Then the homomorphism
		\[
			A \lto \Hom_{alg-gr}(T,\Gm),
		\]
	mapping $a\in A$ to the morphism $\vphi_a:T\to \Gm=\Spec\Z[t,t^{-1}]$ defined by $\vphi_a^{\#}(t)=a$, is an isomorphism of algebraic groups.

\end{corollary}

	Using this, the CC--gadget $\mathcal{L}(X,T):=(\underline{X},X_\C, \ev_X)$ can also be defined in an equivalent way by
	\begin{itemize}
		\item $\underline{X}(D):=\bdu_{i\in I} \Hom_{alg-gr}(G,T_i)$, where $G=\Spec \Z[D]$,
		\item $X_{\C} := X\otz \C$,
		\item for every $i\in I$,
			\[
				\ev_X(D):\Hom(G,T_i) \hookrightarrow \Hom(G\otz \C, T_i\otz \C)\subseteq \Hom(G\otz \C, X_{\C}).
			\]
	\end{itemize}

Furthermore, it follows that for every $i\in I$, we have $A_i\simeq\Z^{\dim T_i}$. Fixing these isomorphisms yields
$$ \underline{X}(D) \ = \ \bdu_{i\in I} \Hom(A_i, D) \ = \ \bdu_{i\in I}D^{\dim T_i} \;. $$

\begin{remark}
We recover Connes-Consani's construction for the CC--gadgets base extending to the multiplicative group $\Gm$ and affine space $\A^n$ 
by the use of the obvious torification $\Gm=\Gm$ for the multiplicative group and $\A^1 = \{0\} \du \left(\A^1\setminus \{0\}\right)$ for the affine line,
respectively, the product torification for higher dimensional affine space (cf.\ paragraphs \ref{subsubsection:Gm} and \ref{subsubsection:A^n}). 
Indeed, let $X$ be the multiplicative group or affine space and $T_i=\Spec\Z[A_i]$ be a torus in the torification as described above.
Let $g\in\Hom(A_i,D)$, 
then $\ev_X(D)(g)\in\Hom(\Spec\C[D],X_\C)$ is determined by $\psi:\left(\Spec \C[D]\right)(\C) \to X_\C(\C)$.
For a character 
$$ \chi \ \in \ \Hom(D,\C^\times) \ \simeq \ \Hom (\Spec \C, \Spec \C[D]) \ = \ (\Spec\C[D])(\C) \;, $$
we have $\psi(\chi):=(\tau_i\otz \C)\left( (\chi(g_j)_{j=1,\dotsc,\dim T_i}\right)$ as in Connes-Consani's description.
\end{remark}

\begin{proposition}
	Let $\Phi=(\vphi,\tilde{\vphi}, \{\vphi_i\}):(X,T)\to (X',T')$ be a torified morphism. The mapping $\mathcal{L}(\Phi):\mathcal{L}(X,T)\to \mathcal{L}(X',T')$ given by $\mathcal{L}(\Phi)=(\underline{\vphi}, \vphi_\C)$ where
	\begin{itemize}
		\item for every finite abelian group $D$, we write $G$ for the group scheme $\Spec \Z[D]$ and for each $i\in I$, we set
			\begin{eqnarray*}
				\underline{\vphi}(D):\Hom_{alg-gr}(G,T_i) & \lto & \Hom_{alg-gr}(G,T_{\tilde{\vphi}(i)}'), \\
				\psi & \lmto & \vphi_i\circ \psi
			\end{eqnarray*}
		\item $\vphi_\C:=\vphi \otz \C:X_\C \to X'_{\C}$
	\end{itemize}
	is a morphism of gadgets.
\end{proposition}

\begin{proof}
The pair $(\underline{\vphi}, \vphi_\C)$ is indeed a morphism of CC--gadgets: the diagram
	\[
		\xymatrix{
			\Hom(G,T_i) \ar[rr]^{\underline{\vphi}(D)} \ar[d]_{\ev_{X}(D)} && \Hom(G,T'_{\tilde{\vphi}(i)}) \ar[d]^{\ev_{X'}(D)} \\
			\Hom (G_\C, X_\C) \ar[rr]_{\vphi_\C} && \Hom (G_\C,X'_\C)
		}
	\]
	commutes since
	\[
		\xymatrix{
			X \ar[r] & X' \\
			T_i \ar[u] \ar[r] & T'_{\tilde{\vphi}(i)} \ar[u]
		}
	\]
	commutes.
\end{proof}

\begin{remark}\hfill
 \label{remark_torified_gadgets}
	\begin{enumerate}
		\item The CC--gadget $\mathcal{L}(X,T):=(\underline{X},X_\C, \ev_X)$ is finite and graded.
		\item $\mathcal{L}(X,T)$ depends in a strong way on the torification, as explained in the following.
	\end{enumerate}
\end{remark}
	
		If $S$ is a second torification of $X$ and $\mathcal{L}(X,S)=(\underline{X}',X'_\C,\ev'_{X})$, we know by Proposition \ref{prop:counting} that there is a bijection between $I$ and $I'$ that respects the grading, so we can assume $I=I'$. It is also clear that $X'_\C = X_\C$. An isomorphism of gadgets $\vphi:\mathcal{L}(X,S)\to \mathcal{L}(X,T)$ consists of a pair $(\underline{\vphi},\vphi_\C)$ where $\underline{\vphi}:X'\Rightarrow X$ is a natural transformation and $\vphi_\C:X_\C\to X_\C$ is a morphism of complex varieties such that for all finite abelian groups $D$ the following diagram commutes:
		\begin{equation}\label{eq:diag1}
			\xymatrix{
				\underline{X}'(D)=\bdu_{i\in I}D^{\dim S_i} \ar[rr]^{\underline{\vphi}(D)} \ar[d]_{\ev'_X(D)} && \bdu_{i\in I}D^{\dim T_i} = \underline{X}(D) \ar[d]^{\ev_X(D)} \\
				\Hom(\Spec(\C[D]),X_\C) \ar[rr]^{\vphi_\C(\C[D])} & &\Hom(\Spec(\C[D]),X_\C).
			}
		\end{equation}

	The morphism $\vphi$ can only be an isomorphism if $\underline{X}'(D)\to \underline{X}(D)$ is a bijection for every finite abelian group $D$. In particular, considering the trivial group $D=\{e\}$ yields a bijection
	\[
		\underline{\vphi}(\{e\}): \bdu_{i\in I} \{e\}^{\dim S_i} \lto \bdu_{i\in I} \{e\}^{\dim T_i},
	\]
	which is merely a bijection $\psi: I \to I$. The trivial group homomorphism $D\to \{e\}$ induces, by the naturality of $\underline{\vphi}$, the commutative diagram
	\[
		\xymatrix{
			\underline{X}'(D) \ar[rr]^{\underline{\vphi}(D)} \ar[d] & & \underline{X}(D) \ar[d]\\
			\underline{X}'(\{e\})\cong I \ar[rr]^{\psi} & & I\cong \underline{X}(\{e\})
		}
	\] 
	and a cardinality argument shows that $\psi$ must respect the grading and that $\underline{\vphi}(D)$ maps $D^{\dim T_i}$ into $D^{\dim S_{\psi(i)}}$. Consequently, commutativity of \eqref{eq:diag1} implies that there are maps $T_i(\C)\to S_{\psi(i)}(\C)$ such that the diagrams
	\[
		\xymatrix{
			T_i(\C) \ar[d]_{\tau_i} \ar[rr] & & S_{\psi(i)}(\C) \ar[d]^{\sigma_{\psi(i)}} \\
			X_\C(\C) \ar[rr]^{\vphi_{\C}} & & X_\C(\C)
		}
	\]
	commute for all $i$.

	For instance,
	\[
		\begin{array}{c}
		\xy 
			(0,0)*{\bullet}; 
			(2,0)*{}; (16,0)*{}; **@{-};
			(0,2)*{}; (0,16)*{}; **@{-};
			(2,2)*{}; (16,2)*{}; **@{--};
			(2,2)*{}; (2,16)*{}; **@{--};
			(16,2)*{}; (16,16)*{}; **@{--};
			(2,16)*{}; (16,16)*{}; **@{--};
    		\endxy
		\end{array}
		= \A^2 =
		\begin{array}{c}
		\xy 
			(16,0)*{\bullet}; 
			(0,0)*{}; (14,0)*{}; **@{-};
			(0,2)*{}; (0,16)*{}; **@{-};
			(2,2)*{}; (16,2)*{}; **@{--};
			(2,2)*{}; (2,16)*{}; **@{--};
			(16,2)*{}; (16,16)*{}; **@{--};
			(2,16)*{}; (16,16)*{}; **@{--};
    		\endxy
		\end{array}
	\]
	are two torifications of the affine plane $\A^2$ that give rise to CC--gadgets that by the above reasoning cannot be isomorphic. This illustrates (2).
	
	\begin{theorem} \label{thm_ll}
		Given an affinely torified variety $(X,T)$, the corresponding CC--gadget $\mathcal{L}(X,T)=(\underline{X}, X_\C, \ev_X)$ is a CC--variety over $\Fun$ such that $X\otf \Z\cong X$. If $f:(X,T)\to (X',T)$ is a torified morphism then $\mathcal{L}(f)$ is a morphism of CC--varieties. More precisely, the composition of the functors $\mathcal L$ and $-\otimes_\Fun\Z$ is isomorphic to the functor from affinely torified varieties to schemes that forgets the torification.
	\end{theorem}
	\begin{proof}
		Start assuming that $X$ is affine. Let $\mathcal{G}(X)$ be the gadget defined by $X$ as in Definition \ref{definition_CC-variety}, and define the immersion
			\[
				i:\mathcal{L}(X,T) \lto \mathcal{G}(X)
			\]
		as follows:
		\begin{enumerate}
			\item For every finite abelian group $D$, the map
				\[
					\underline{i}(D):\underline{X}(D)=\bdu_{i\in I} D^{\dim T_i} \lto \Hom(\Spec\Z[D],X)
				\]
			is defined in the same way as the evaluation map $\ev_X$, using the fact that this map is obtained by extension of scalars. In other words, we have the commutative diagram 
			\begin{equation}\label{eq:diag2}
				\xymatrix@R=0.6pc{
					\underline{X} \ar[rr]^{\ev_X} \ar[dr]_-{\underline{i}} && \Hom(\Spec \C[-],X_\C). \\
					& \Hom(\Spec \Z[-],X) \ar[ur]_{-\otz \C}
				}
			\end{equation}
			It is clear that $\underline{i}(D)$ is injective for every $D$.
			\item The morphism of varieties $i_\C:X_\C\to X_\C$ is the identity.
			\item By the commutativity of \eqref{eq:diag2}, the diagram
				\[
					\xymatrix{
						\underline{X}(D) \ar[rr]^{\underline{i}(D)} \ar[d]_{\ev_X} && \Hom(\Spec \Z(D],X) \ar[d]^{-\otz \C} \\
						\Hom(\Spec \C[D],X_\C) \ar[rr]_{\Id} && \Hom(\Spec \C[D],X_\C)
					}
				\]
			commutes for every finite abelian group $D$.		
		\end{enumerate}
		To verify the universal property, let $V$ be an affine reduced scheme of finite type over $\Z$, and $(\underline{\vphi},\vphi_\C): \mathcal{L}(X,T) \to \mathcal{G}(V)$ a morphism of gadgets. We need to find a morphism $\vphi:X\to V$ of schemes such that the diagram
			\[
				\xymatrix{
					\mathcal{L}(X,T) \ar[r]^{i} \ar[dr]_-{(\underline{\vphi},\vphi_\C)} & \mathcal{G}(X) \ar[d]^{\mathcal{G}(\vphi)} \\
					& \mathcal{G}(V)
				}
			\]
		commutes. Since $\vphi_\C:X_\C\to V_\C=V \otz \C$ is already given, it suffices to prove that there is a morphism $\vphi:X\to V$ such that $\vphi_\C=\vphi\otz \C$, or in other words: we have to show that $\vphi_\C$ is already defined over $\Z$.
		
		The map $\vphi_\C$ is defined over $\Z$ if ${\vphi_\C}_{|Y}$ is defined over $\Z$ for every irreducible component $Y$ of $X$. To each irreducible component of $X$ corresponds a unique open torus $T_i\subseteq Y$, which is the torus that contains the generic point of $Y$ (see Lemma \ref{lemma:torificationpoints}, (2)). Since $T_i=\Gm^{\dim T_i}$ is a CC--variety over $\Fun$ (cf. \cite[Sect. 3.1]{Connes2008}), the map ${\vphi_\C}_{|T_i}$ is defined over $\Z$, and thus $\vphi_\C$ is a rational function with $\vphi_{|Y}:Y\to V$ defined over $\Z$. Consequently, $\xymatrix{\vphi:X \ar@{-->}[r] & V}$ is a rational function defined over $\Z$.
		
		In order to show that $\vphi$	is indeed a morphism, we have to show that for all affine open $Z\subseteq X$ and $U\subseteq V$ such that $\varphi_\C: Z_\C \to U_\C$, and for all $h\in \O_U(U)$, we have $\vphi^{\#}(h)\in \O_Z(Z)$. 
		
		We know that $\vphi^{\#}(h)\in \O_Z\left(Z\cap \left(\bigcup_{i\in I^o}T_i \right)\right)$, where $I^o=\{i\in I|\ T_i\ \text{is open in $X$}\}$. If we denote by $I^{cl}:=I\setminus I^o$, there is some $\delta\in \O_Y(Y)$ such that $Z\cap \left(\bigcup_{i\in I^{cl}} T_i\right)$ is contained in the vanishing set of $\delta$, and thus $\vphi^{\#}(h)\in \O_Z(Z)[\delta^{-1}]$. 
		
		But we also know that $\vphi_\C^{\#}(h)\in \O_{Z_\C}(Z_\C)=\O_Z(Z)\otz \C$. Since 
			\[
				\O_Z(Z)[\delta^{-1}]\cap \left( \O_Z(Z)\otz \C \right) = \O_Z(Z),
			\]
		 where we can consider all sets as subsets of the function field $F(Z_\C)$, we obtain $\vphi^{\#}(h)\in \O_Z(Z)$, proving the desired result. Note that Proposition \ref{prop:counting} implies that $\uline X$ represents the counting function of $X$.

		 For the general case, let $\{U_i\}$ be the collection of all affine open subschemes of $X$ such that $T$ restricts to a torification $T_i$ of $U_i$. Then $X_i=\mathcal{L}(U_i,T_i)$ is an affine CC--variety, and by the definition of a general CC--variety, $X$ is a CC--variety. Furthermore, $\{X_i\}$ is the family of all affine open subschemes of $\mathcal{L}(X,T)$, thus $\mathcal L(X,T)_\Z$ is defined as the direct limit over the family of the $X_{i,\Z}\simeq U_i$, which is nothing else than $X$ itself.

Note that the defining property of a morphism between CC--varieties follows from the fact that torified morphisms between affinely torified varieties are affinely torified (see Lemma \ref{lemma: affinely torified morphisms}). Concerning functoriality, it is clear that for all torified morphisms $f: (X,T)\to (X',T')$ between torified varieties, the diagram
 $$ \xymatrix{X\ar[rr]^f\ar[d]^\sim && X'\ar[d]^\sim\\\mathcal{L}(X,T)_\Z\ar[rr]^{\mathcal L(f)}&&\mathcal{L}(X',T')_\Z} $$
 commutes. This establishes an isomorphism between the composition of $\mathcal L$ and $-\otimes_\Fun\Z$ 
 and the forgetful functor from affinely torified varieties to schemes.
	\end{proof}

\begin{remark}
 In the proof of this theorem, we made only use of the highest degree term $\uline X^{(\dim X)}$ of the functor $\uline X=\{\uline X^{(l)}\}$
 in the proof that $\Gm^{\dim X}$ is an affine CC-variety. 
 This has the following consequence: Let $X$ be a reduced scheme of finite type over $\Z$ with an open affine cover $\{U_i\}$ 
 and with an open subscheme $T$ that is isomorphic to $\Gm^{\dim X}$ such that $T\subset U_i$ for all $i$.
 Define $\uline X(D)=\uline Y_i(D)=D^{\dim X}$, and $X_\C=X\otimes_\Z\C$ and $Y_{i,\C}=U_i\otimes_\Z\C$ for all $i$.
 Define the evaluations $\ev_X$ and $\ev_i$ in the same way as for $\mathcal L$.
 Then the same proof as above shows that $(\uline X,X_\C,\ev_X)$ is a variety over $\F_1$ (in the sense of Connes and Consani, cf.\ \cite[section 3.4]{Connes2008})
 covered by the affine varieties $(\uline Y_i,Y_{i,\C},\ev_i)$ over $\F_1$.
 \end{remark}


\section{Soul\'e's geometry.}\label{section:soule}

\subsection{S--objects and S--varieties}
 In this section we recall some notions of $\Fun$-geometry introduced by Soul\'e in \cite{Soule2004}, reformulated as in \cite[Section 2.2]{Connes2008}.

\begin{definition}
  Let $\mathcal{R}$ the category of commutative rings which are finite and flat as $\Z$--modules. A \dtext{(Soul�) gadget} (S--gadget for short) 
  over $\Fun$ is a triple $X=(\underline{X},\mathcal{A}_X,e_X)$ consisting of
  \begin{itemize}
   \item a functor $\underline{X}:\mathcal{R}\to \Sets$,
   \item a complex algebra $\mathcal{A}_X$,
   \item a natural transformation $e_X:\underline{X}\Longrightarrow \Hom (\mathcal{A}_X, -\otz \C)$. 
  \end{itemize} 
  An S--gadget $X$ is \dtext{finite} if for all $R\in \mathcal{R}$ the set $\underline{X}(R)$ is finite. A \dtext{morphism} $\vphi:X\to Y$ of 
  S--gadgets is a couple $(\underline{\vphi}, \vphi^\ast)$, where $\underline{\vphi}:\underline{X}\Rightarrow \underline{Y}$ is a 
  natural transformation and $\vphi^\ast:\mathcal{A}_Y\to \mathcal{A}_X$ is morphism of algebras such that
  \[\xymatrix{\underline{X}(R) \ar[rr]^{\underline{\vphi}(R)} \ar[d]_{e_X(R)} && \underline{Y}(R) \ar[d]^{e_Y(R)} \\
	      \Hom(\mathcal{A}_X, R\otz\C) \ar[rr]^{\vphi^\ast(R\otimes_\Z\C)} && \Hom(\mathcal{A}_Y, R\otz\C)}	\]
  commutes for all $R\in \mathcal{R}$. If $\vphi^\ast$ is injective and $\underline{\vphi}(R)$ is injective for all $R\in \mathcal{R}$, 
  we say that $\vphi$ is an \dtext{immersion}. 
\end{definition}
	
 We can associate a gadget $\mathcal{T}(V)=(\underline{V},\mathcal{O}_{V_\C}(V_\C), e_V)$ to any scheme of finite type $V$ over $\Z$, 
 where $\underline{V}(R):=\Hom(\Spec R, V)$ is the functor of points, $\mathcal{O}_{V_\C}(V_\C)$ the algebra of global sections 
 of the complexification of $V$, and $e_V$ is the extension of scalars to $\C$. 
 For a morphism $f: U\to V$, we define $\mathcal T (f):\mathcal T(U)\to\mathcal T(V)$ as the pair $(\uline f,f_\C^\#)$
 where $\uline f(R):\uline U(R)\to \uline V(R)$ is the induced morphism on sets of points and $f_\C^\#:\mathcal A_V\to\mathcal A_U$ 
 is the complexification of the morphism between global sections. It is immediate that $\mathcal T(f)$ is a morphism.
 Thus $\mathcal T$ is a functor from schemes of finite type over $\Z$ to S--gadgets.

\begin{definition}
 An \dtext{affine (Soul\'e) variety over $\Fun$} (affine S--variety for short) is a finite S--gadget $X$ 
 such that there is an affine scheme $X_\Z$ of finite type over $\Z$ and an immersion of gadgets $i_X:X\to \mathcal{T}(X_\Z)$ 
 satisfying the following universal property:
 For every affine scheme $V$ of finite type over $\Z$ and every morphism of S--gadgets $\vphi:X\to \mathcal T(V)$ there is a unique morphism of schemes $\vphi_\Z:X_\Z\to V$ 
 such that $\vphi=\mathcal{T}(\vphi_\Z)\circ i_X$.
\end{definition}
 
 We define the category of affine S--varieties as the full subcategory of S--gadgets whose objects are affine S--varieties.
 The universal property defines the base extension functor from affine S--varieties to affine schemes over $\Z$.
 Namely, it sends $X$ to $X_\Z$ and a morphism $\vphi: X\to Y$ to $(i_Y\circ\vphi)_\Z: X_\Z \to Y_\Z$ (cf.\ Lemma \ref{lemma_univ_prop} below).

 By \cite[Proposition 2]{Soule2004}, the functor $R\mapsto \mathcal{T}(\Spec(R))$ 
 is a fully faithful embedding of the category $\mathcal{R}^{op}$ into the category of affine S--varieties.
 
\begin{definition}
 An \dtext{(Soul\'e) object over $\Fun$} (S--object) is a triple $X=(\uuline{X}, \mathcal{A}_X, e_X)$ consisting of
 \begin{itemize}
  \item a \emph{contravariant} functor $\uuline{X}:\{\textrm{Affine S--varieties}\} \to \Sets$,
  \item a complex algebra $\mathcal{A}_X$,
  \item a natural transformation $e_X:\uuline{X}\Longrightarrow \Hom (\mathcal{A}_X, \mathcal A_{(-)})$.
 \end{itemize}
 An S--object is \dtext{finite} if $\uuline{X}(\mathcal{T}(\Spec R))$ is finite for all $R\in \mathcal{R}$. 
 A \dtext{morphism of objects} $\vphi:X\to Y$ is given by a natural transformation $\uuline{\vphi}:\uuline{X}\Rightarrow \uuline{Y}$ 
 and a morphism of algebras $\vphi^\ast:\mathcal{A}_Y\to \mathcal{A}_X$ such that
  \[ 
   \xymatrix{
    \uuline{X}(V) \ar[rr]^{\uuline{\vphi}(V)} \ar[d]_{e_X(V)} && \uuline{Y}(V) \ar[d]_{e_Y(V)} \\
    \Hom(\mathcal{A}_X,\mathcal{A}_V) \ar[rr]^{\vphi^\ast(V\otimes_\Z\C)} && \Hom(\mathcal{A}_Y,\mathcal{A}_V)
   }
  \]
 commutes for all $V\in \mathcal{A}$. If $\vphi^\ast$ and $\uuline{\vphi}(V)$ are injective for all $V\in \mathcal{A}$, 
 then we say that $\vphi$ is an \dtext{immersion of objects}.
\end{definition}
 
 We can associate an object $\mathcal{O}b(S)=(\uuline{S},\mathcal{O}_{S_\C}(S_\C),e_S)$ to any scheme $S$ of finite type over $\Z$
 via $\uuline{S}(V):=\Hom(V_\Z,S)$ and evaluation $e_S(x)$ defined by the composition 
 \[
  \xymatrix{
   \mathcal{O}_{S_\C}(S_\C) \ar[r]^{x^\ast} & \mathcal{O}_{V_\C}(V_\C) \ar[r]^{i^\ast} & \mathcal{A}_V
  }.
 \]  
   
\begin{definition}
 A \dtext{(Soul\'e) variety over $\Fun$} (S--variety) is a finite S--object $X$ for which there exists a scheme $X_\Z$ of finite type over $\Z$ 
 and an immersion $i:X\to \mathcal{O}b(X_\Z)$ such that for every scheme $V$ of finite type over $\Z$ 
 and every morphism of objects $\vphi:X\to \mathcal{O}b(V)$, 
 there is a unique morphism of schemes $\vphi_\Z:X_\Z\to \mathcal{O}b(V)$ such that $\vphi= \mathcal{O}b(\vphi_\Z)\circ i$. 
\end{definition}

 We define the category of S--varieties as the full subcategory of S--objects whose objects are S--varieties.
 An S--gadget can be considered as an S--object in the following way. If $X=(\uline X,\mathcal A_X,e_X)$ is an S--gadget, 
 then the associated S--object is $(\uuline X, \mathcal A_X,e'_X)$, where for an affine S--variety A, 
 $$ \uuline X(A) \ = \ \Hom(A,X) $$ 
 and $e'_X$ sends $\vphi=(\uline\vphi,\vphi^\ast)\in\uuline X(A)$ to $\vphi^\ast: \mathcal A_X \to \mathcal A_A$.
 This defines a fully faithful functor from S--gadgets to S--objects.
 The essential image (i.e. the smallest subcategory that contains all isomorphisms with an object of the image) of the category of affine S--varieties is the full subcategory of S--varieties whose objects base extend to an affine scheme over $\Z$
 (cf.\ \cite[section 4.2, Prop.\ 3]{Soule2004}).

 An immediate observation following from the definition of an S--variety is the following.
 
\begin{lemma} \label{lemma_univ_prop}
Let $X$ be an S--variety and $V$ a scheme of finite type over $\Z$.
 \begin{enumerate}
  \item The map $\vphi\mapsto\vphi_\Z$ given by the universal property of $X$ defines a bijection 
   \[
    \Hom(X,\mathcal{O}b(V)) \ \lto \ \Hom(X_\Z,V)\;.
   \]
  \item If $\iota: Y\hookrightarrow\mathcal{O}b(V)$ is an immersion of S--objects, then $\vphi\mapsto(\iota\circ\vphi)_\Z$ defines an embedding 
   \[
    \Hom(X,Y) \ \longrightarrow \ \Hom(X_\Z,V)\;.
   \] 
    In particular, if $X$ and $Y$ are both S--varieties, then $\Hom(X,Y)\hookrightarrow\Hom(X_\Z,Y_\Z)$. 
 \end{enumerate}
\end{lemma}


\subsection{Smooth toric varieties as S--varieties}
\label{toric_S-objects}
 Soul\'e describes in \cite[section 5.1]{Soule2004} an S--object $\mathcal S(X)$ associated to a toric variety $X$.
 Note that this association works for arbitrary toric varie\-ties,
 though Soul\'e proves only for smooth toric varieties $X$ that $\mathcal S(X)$ is an S--variety.
 Further note that we are working with an different complex algebra than Soul\'e does,
 but that results transfer by \cite[Prop.\ 4]{Soule2004}.
 Given a toric variety $X$ with fan $\Delta$, 
 we define the S--object $\mathcal S(X)$ in two steps.

 In the first step, we define for every cone $\tau\in\Delta$ the S--gadget $X_\tau=(\uline X_\tau,\mathcal A_{X_\tau}, e_{X_\tau})$
 as follows. Let $A_\tau=\tau^\vee\cap(\Z^n)^\vee$ be as in paragraph \ref{section_toric_as_torified}. Let $\mu(R)$ be the roots of unity of the ring $R$.
 For every $R\in\mathcal R$, put $\uline X_\tau(R)=\Hom(A_\tau,\mu(R)_0)$, 
 the set of semi-group homomorphisms from $A_\tau$ to 
 the multiplicative semi-group $\mu(R)_0=\{0\}\cup\mu(R)$.
 Put $\mathcal A_{X_\tau}=\C[A_\tau]$ and let 
 $$ e_{X_\tau}(R): \ \uline X_\tau(R) = \Hom(A_\tau,\mu(R)_0) \ \longrightarrow \ \Hom(\C[A_\tau],R\otimes_\Z\C) $$
 be the natural map.
 
 For $U_\tau=\Spec\Z[A_\tau]\subset X$, we have a canonical morphism of S--gadgets $\iota_\tau: X_\tau\hookrightarrow\mathcal T(U_\tau)$,
 which is an immersion since the complex algebras are the same and since for every $R\in\mathcal R$, we have 
 $$ \uline X_\tau(R) \ = \ \Hom(A_\tau,\mu(R)_0) \ \subset \ \Hom(\Z[A_\tau],R) \ = \ \uline U_\tau(R) \;. $$
 Consequently, the universal property of an affine S--variety $V$ with immersion $\iota_V:V\to\mathcal T(V_\Z)$
 implies that given a morphism $\varphi: V\to X_\tau$ of S--gadgets there is a unique morphism $\varphi_\Z$ such that
 $\iota_\tau\circ\varphi=\mathcal T(\varphi_\Z)\circ\iota_V$. 
 By Lemma \ref{lemma_univ_prop}, we obtain inclusions
 $$ \Hom(V,X_\tau)\subset\Hom(V_\Z,U_\tau)\subset\Hom(V_\Z,X) \;. $$

 In the second step, we define the S--object $\mathcal S(X)=(\uuline X,\mathcal A_X,e_X)$ as follows.
 For every affine S--variety $V$, put 
 $$ \uuline X(V)=\bigcup_{\tau\in\Delta} \Hom(V,X_\tau), $$
 where the union is taken in $\Hom(V_\Z,X)$.
 Put $\mathcal A_X=\O_{X_\C}(X_\C)$, where $X_\C=X\otimes_\Z\C$, and let $e_X(V):\uuline X(V)\subset\Hom(V_\Z,X)\to\Hom(\mathcal A_X,\mathcal A_V)$
 be the natural map.
 
 In a natural way, $\mathcal S$ extends to a functor from toric varieties to S--objects. 
 Given a toric morphism $f:X\to X'$ that is induced by a morphism of cones $\delta:\Delta\to\Delta'$ (see paragraph \ref{section_toric_as_torified}),
 then following the constructions of the first step yields morphisms of S--gadgets $f_\tau: X_\tau\to X'_{\delta(\tau)}$ for every $\tau\in\Delta$. 
 In the second step, taking the union over all cones $\tau\in\Delta$ defines a morphism $\mathcal S(f):\mathcal S(X)\to\mathcal S(X')$.

 As a consequence of \cite[Thm.\ 1(i)]{Soule2004} and \cite[Prop.\ 4]{Soule2004}, we obtain the following result.

\begin{theorem}[Soul\'e] \label{thm_soule}
 Let $X$ be a smooth toric variety. Then the S--object $\mathcal S(X)$ is an S--variety such that $X\simeq\mathcal S(X)\otimes_{\F_1}\Z$.
\end{theorem}

\begin{remark}
 In particular, \cite[Prop.\ 3]{Soule2004} implies that
 the S--gadgets $X_\tau$ are affine S--varieties with $U_\tau\simeq X_\tau\otimes_{\F_1}\Z$ for all $\tau\in\Delta$,
\end{remark}


\subsection{Affinely torified varieties as S--varieties}
\label{affinely_torified_S-varieties}

 In this section, we define a functor $\mathcal S^\sim$ from the category of affinely torified varieties to
 the category of S--objects, prove that $\mathcal S^\sim$ extends $\mathcal S$, 
 which allows us to drop the superscript ``$\sim$'',
 and show that Soul\'e's result (Theorem \ref{thm_soule}) extends to this class of S--objects.

 Let $X$ be torified variety with an affine torification $T=\{T_i\hookrightarrow X\}_{i\in I}$.
 Put $A_i=\Hom_{alg-gr}(T_i,\Gm)$ for $i\in I$.
 Let $\{U_j\}_{j\in J}$ be the \emph{maximal torified atlas}, i.e.\ the family of all affine open subschemes $U_j$ of $X$ 
 such that $U_j=\decomp_{i\in I_j} T_i$ for a subset $I_j$ of $I$. 
 We define an S--object $\mathcal S^\sim(X,T)$ in two steps.

 In the first step, we define an S--gadget $X_j^\sim=(\uline X_j^\sim,\mathcal A_j^\sim, e_j^\sim)$ for every $j\in J$ as follows.
 For $R\in\mathcal R$, put $\uline X_j^\sim(R)=\coprod_{i\in I_j}\Hom(A_i,\mu(R))$.
 Put $\mathcal A_j^\sim=\O_{U_{j,\C}}(U_{j,\C})$ and let
  \[
   e_j^\sim(R): \ \coprod_{i\in I_j}\; \Hom(A_i,\mu(R)) \ \longrightarrow \ \coprod_{i\in I_j}\; \Hom(\C[A_i],R\otimes\C) \ \hookrightarrow \ \Hom(\mathcal A_j^\sim,R\otimes\C) \;, 
 \]
 be the composition of the natural maps $\Hom(A_i,\mu(R))\to\Hom(\C[A_i],R\otimes\C)$ with the inclusion induced by the restriction maps
  \[ 
  \mathcal A_j^\sim \ = \ \O_{U_{j,\C}}(U_{j,\C}) \ \longrightarrow \  \O_{U_{j,\C}}(T_{i,\C}) \ = \ \C[A_i] \;.
 \]
 
 For every $i\in I$, there is a canonical morphism of S--gadgets $\iota_j: X_j^\sim\hookrightarrow\mathcal T(U_j)$,
 which is an immersion since the complex algebras are the same and since for every $R\in\mathcal R$, we have 
 $$ \uline X_j^\sim(R) \ = \ \coprod_{i\in I_j}\; \Hom(A_i,\mu(R)) \ \subset \ \coprod_{i\in I_j}\;\Hom(\Z[A_i],R) \ = \ \uline U_j(R) \;. $$
 Consequently, we obtain inclusions
 $$ \Hom(V,X_j^\sim)\subset\Hom(V_\Z,U_j)\subset\Hom(V_\Z,X) $$
 for every affine S--variety (cf.\  Lemma \ref{lemma_univ_prop}).
 
 In the second step, we define the S--object $\mathcal S^\sim(X,T)=(\uuline X,\mathcal A_X,e_X)$ as follows.
 For every affine S--variety $V$, put
 $$ \uuline X(V) \ = \ \bigcup_{j\in J} \; \Hom(V,X_j^\sim) \;, $$
 where the union is taken in $\Hom(V_\Z,X)$.
 Put $\mathcal A_X=\O_{X_\C}(X_\C)$, and let 
 $$ e_X(V): \ \uuline X(V) \ \subset \ \Hom(V_\Z,X) \ \to \ \Hom(\mathcal A_X,\mathcal A_V) $$
 be the natural map.

 By following through the construction of $\mathcal S^\sim$, we can associate to every torified morphism a morphism of S--objects.
 Note that in the second step, we have to make use of Lemma \ref{lemma: affinely torified morphisms}: 
 Let $f:X\to X'$ be a torified morphism. Then the family of all opens $U$ of $X$ in the maximal torified atlas of $X$ 
 whose image under $f$ is affine covers $X$. This allows us to define $\mathcal S^\sim(f)$ as the union of the restrictions 
 $\mathcal S^\sim(f_{|U})$ to affine opens. 
 Thus we defined $\mathcal S^\sim$ as a functor from the category of affinely torified varieties to the category of S--objects.

\begin{remark}
 \label{remark_torified_objects}
 A discussion similar to the one in Remark \ref{remark_torified_gadgets} shows that different affine torifications of the same 
 torified variety $X$ can lead to non-isomorphic S--objects. 
 The two torifications of $\A^2$ given in Remark \ref{remark_torified_gadgets} provide an example.
\end{remark}

 We show that $\mathcal S^\sim$ extends indeed Soul\'e's functor $\mathcal S$.
 Let $X$ be a toric variety with fan $\Delta$. Let $T_\Delta$ be the torification of $X$ as defined in section \ref{section_toric_as_torified}.
 We put $\mathcal S^\sim(X)=\mathcal S^\sim(X,T_\Delta)$.
 
\begin{lemma}
 For every $\tau\in\Delta$, there are isomorphisms 
 \[
  \xymatrix{
   \coprod_{\sigma\subset\tau}\; \Hom(A_\sigma^\times,\mu(R))\ar@<0.5ex>[rrr]^{\alpha(R)} 
             &&& \Hom(A_\tau,\mu(R)_0) \ar@<0.5ex>[lll]^{\beta(R)}
        }
    \]
 that are functorial in $R\in\mathcal R$.
\end{lemma} 
 
\begin{proof}
 We construct the maps $\alpha=\alpha(R)$ and $\beta=\beta(R)$ as follows.
 Let $\vphi:A_\sigma^\times\to \mu(R)$ be an element of $\Hom(A_\sigma^\times,\mu(R))$.
 Since $A_\tau\subset A_\sigma$, we can define $\psi=\alpha(\vphi)$ by
 $$\begin{array}{cccl} \psi: & A_\tau & \longrightarrow & \mu(R)_0\;. \\
                                   &    a  & \longmapsto     & \left\{ \begin{array}{ll}\vphi(a)& \text{if }a\in A_\sigma^\times\cap A_\tau\\
                                                                  0       & \text{otherwise}
                            \end{array} \right. \end{array} $$
 
 Let $\psi: A_\tau\to\mu(R)_0$ be an element of $\Hom(A_\tau,\mu(R)_0)$.
 We claim that there is a \emph{smallest cone for $\psi$}, i.e.\ 
 a smallest subcone $\sigma$ of $\tau$ such that $\psi$ extends to a morphism $\psi': A_\sigma\to\mu(R)_0$ of semi-groups.
 Indeed, assume that $\psi$ extends to $\psi_1: A_{\tau_1}\to\mu(R)_0$ and $\psi_2: A_{\tau_2}\to\mu(R)_0$ for two cones $\tau_1,\tau_2\subset\tau$.
 Then $\psi$ extends also to a morphism from the semi-group generated by $A_{\tau_1}$ and $A_{\tau_2}$. 
 But this semi-group is nothing else than $A_{\tau_1\cap\tau_2}$. This proves the claim.
 
 If $\sigma$ is the smallest cone for $\psi$,
 then define $\vphi=\beta(\psi)$ as the restriction of $\psi':A_\sigma\to\mu(R)_0$ to $A_\sigma^\times$.

 We show that $\alpha$ and $\beta$ are mutually inverse.
 Let $\vphi:A_\sigma^\times\to \mu(R)$ be an element of $\Hom(A_\sigma^\times,\mu(R))$ and $\psi=\alpha(\vphi):A_\tau\to\mu(R)_0$.
 Then $\sigma$ is the smallest cone for $\psi$
 since for every $\sigma'\subsetneq\sigma$, the larger semi-group $A_{\sigma'}$ is still generated by $A_\sigma$
 and consequently $A_\sigma^\times\subsetneq A_{\sigma'}^\times$, but we know that $(\psi_Z^\#)^{-1}(\mu(B))=A_\sigma^\times$.
 Thus $\beta(\psi)$ equals $\vphi$ by definition of $\psi$. 

 Let conversely $\psi: A_\tau\to\mu(R)_0$ be an element of $\Hom(A_\tau,\mu(R)_0)$ and $\vphi=\beta(\psi)\in\Hom(A_\sigma^\times,\mu(R))$,
 where $\sigma$ is the smallest cone for $\psi$. It is clear by definition that $\alpha(\vphi)$ equals $\psi$ 
 restricted to $A_\sigma^\times\cap A_\tau$. We have to show that $\psi(A_\sigma\setminus A_\sigma^\times)=\{0\}$,
 where we extended $\psi$ to $\psi:A_\sigma\to\mu(R)_0$.
 If there is an $a\in A_\sigma\setminus A_\sigma^\times$ such that $\psi_\Z^\#(a)\neq0$, i.e.\ $\psi_\Z^\#(a)\in\mu(B)$,
 we derive a contradiction to the minimality of $\sigma$ as follows.
 
 Choose a basis $(\lambda_i)_{i\in N}$ of $\R^n$, where $N=\{1,\dotsc,n\}$ and $n$ is the dimension of $X$,
 such that $\sigma=\langle\lambda_i\R_{\geq0}\rangle_{i\in S}$ for some $S\subset N$ 
 and $\langle\lambda_i\R\rangle_{i\in N\setminus S}$ is orthogonal to $\sigma$
 (here ``$\langle-\rangle$'' denotes the generated semi-group in $\R^n$).
 Let $(\lambda_i^\ast)_{i\in N}$ be the dual basis of $(\lambda_i)_{i\in N}$,
 then $\sigma^\vee=\langle l_i\R_{\geq0} \rangle_{i\in S}+\langle l_i\R \rangle_{i\in N\setminus S}$.
 The set $\{\sigma'\in\Delta\mid\sigma'\subset\sigma\}$ is the set of cones of the form $\sigma_J=\langle\lambda_i\R_{\geq0}\rangle_{i\in J}$, 
 where $J$ is a subset of $S$. 
 For every $i\in N$, define $l_i$ as the smallest multiple of $\lambda_i^\ast$ such that $l_i\in A_\sigma$.
 Then $\sigma^\vee_J=\langle l_i\R_{\geq0} \rangle_{i\in J}+\langle l_i\R \rangle_{i\in N\setminus J}$ for every $J\subset S$,
 and the semi-group $L_J=\langle l_i \rangle_{i\in J}+\langle l_i\Z \rangle_{i\in N\setminus J}$ is of finite index in $A_{\sigma_J}$.
 This implies that for the chosen $a\in A_\sigma\setminus A_\sigma^\times$, a positive multiple $m\cdot a$ is in $L_S$,
 i.e.\ $m\cdot a=\sum_{i\in S} c_i l_i$ for certain non-negative integers $c_i$.
 Since we assume that $\psi_\Z^\#(a)\in\mu(B)$, we have that $\sum_{i\in S}c_i\psi_\Z^\#(l_i)=\psi_\Z^\#(m\cdot a)=\psi_\Z^\#(a)^m\neq0$ and thus
 already $\psi_\Z^\#(l_i)\neq0$ for some $i\in S$. Put $J=S\setminus\{i\}$. 
 Then $\psi_\Z^\#$ can be extended to a semi-group morphism $\tilde\psi_\Z^\#: A_{\sigma_J}\to\mu(B)_0$, 
 which yields the desired contradiction to the minimality of $\sigma$. 
 This completes the proof that $\alpha$ and $\beta$ are mutually inverse.

 It is clear that $\alpha(R)$ and $\beta(R)$ are functorial in $R$, i.e.\ that for every morphism $f:R_1\to R_2$ in $\mathcal R$,
 the diagram
 \[ \xymatrix{ \coprod\limits_{\sigma\subset\tau}\; \Hom(A_\sigma^\times,\mu(R_1))\ar@<0.8ex>[rr]^{f_\ast} \ar[d]_{\alpha(R_1)} 
                  &&\coprod\limits_{\sigma\subset\tau}\;  \Hom(A_\sigma^\times,\mu(R_2)) \ar[d]_{\alpha(R_2)} \\
       \Hom(A_\tau,\mu(R_1)_0)\ar[rr]^{f_\ast} && \Hom(A_\tau,\mu(R_2)_0)}  \]
 commutes.
\end{proof}
 
\begin{proposition}
\label{prop_S_and_S_tilde_are_isomorphic}
 The functors $\mathcal S$ and $\mathcal S^\sim$ from the category of toric varieties to the category of S--objects are isomorphic.
\end{proposition}

\begin{proof}
 Let $X$ be a toric variety with fan $\Delta$.
 Then the maximal torified atlas of $(X,T_\Delta)$ is $\{U_\tau\}_{\tau\in\Delta}$.
 We first show that for every $\tau\in\Delta$, the corresponding S--gadgets $X_\tau$ 
 and $X_\tau^\sim$ are isomorphic.
 We define maps 
 \[\xymatrix{X_\tau^\sim=(\uuline X^\sim,\mathcal A_X^\sim,e_X^\sim)\ar@<1ex>[rrr]^{\alpha_\tau=(\uline\alpha_\tau,\alpha_{\tau,\C})} 
             &&& X_\tau=(\uuline X,\mathcal A_X,e_X)\;. \ar@<0.0ex>[lll]^{\beta_\tau=(\uline\beta_\tau,\beta_{\tau,\C})}   }\]
 as follows. For every $R\in\mathcal R$, define $\uline\alpha_\tau(R)$ as the map $\alpha(R)$ of the previous lemma. 
 Define $\alpha_{\tau,\C}$ as the identity map of $\mathcal A_\tau^\sim=\O_{X_{\tau,\C}}=\mathcal A_\C$.
 Concerning $\beta_\tau$, define$\uline\beta_\tau(R)$ as the map $\beta(R)$ of the previous lemma for every $R\in\mathcal R$.
 Define $\beta_{\tau,\C}$ like $\alpha_{\tau,\C}$ as the identity map.
 It is easily verified that $\alpha_\tau$ and $\beta_\tau$ are indeed morphisms of S--gadgets.
 The previous lemma implies that $\alpha_\tau$ and $\beta_\tau$ are inverse to each other.

 Since the second steps in the constructions of $\mathcal S(X)$ and $\mathcal S^\sim(X)$ coincide,
 the families $\{\alpha_\tau\}_{\tau\in\Delta}$ and  $\{\beta_\tau\}_{\tau\in\Delta}$ define mutually inverse
 morphisms $\alpha_X:\mathcal S^\sim(X)\to\mathcal S(X)$ and $\beta_X:\mathcal S(X)\to\mathcal S^\sim(X)$ of S--objects.
 It is straightforward to verify that $\alpha_X$ and $\beta_X$ are functorial in $X$, i.e.\ that the diagram
 \[ \xymatrix{\mathcal S(X_1) \ar[rr]^{\mathcal S(f)} \ar[d]_{\beta_{X_1}} && \mathcal S(X_2) \ar[d]_{\beta_{X_2}} \\
      \mathcal S^\sim(X_1) \ar[rr]^{\mathcal S^\sim(f)} && \mathcal S^\sim(X_2)}  \]
 commutes for every toric morphism $f:X_1\to X_2$.
 Thus we established an isomorphism of functors.
\end{proof}

 The proposition justifies that we can write $\mathcal S(X,T)=\mathcal S^\sim(X,T)$ for an affinely torified variety $(X,T)$.

\begin{theorem}
 \label{thm_javier}
 If $(X,T)$ is an affinely torified variety,
 then $\mathcal S(X,T)$ is an S--variety such that $\mathcal S(X,T)_\Z\simeq X$.
 More precisely, the composition of the functors $\mathcal S$ and $-\otimes_\Fun\Z$ is isomorphic to the functor 
 from affinely torified varieties to schemes that forgets the torification.
\end{theorem}

\begin{proof}
 Define the morphism of S--objects $\iota=(\uuline\iota,\iota_\C):\mathcal S(X,T)\to\mathcal T(X)$ as follows.
 Write $\mathcal S(X,T)=(\uuline X,\mathcal A_X,e_X)$.
 For every affine S--variety $V$, let $\uuline\iota(V): \uuline X(V)\hookrightarrow\Hom(V_\Z,X)$ be the extension of scalars, 
 which is an injective map (cf.\ Lemma \ref{lemma_univ_prop}).
 Let $\iota_\C$ be the identity map of $\mathcal A_X=\O_{X_\C}(X_\C)$.
 It is clear that $\iota$ defines a morphism and that it is an immersion of S--objects.
 
 We raise in three steps the generality of $X$. 
 In the first step, let $X$ be ${\mathbb G}_m^n$ for an $n\geq0$. 
 Then there exists up to isomorphism only one torification of ${\mathbb G}_m^n$, namely $T=\{{\mathbb G}_m^n\to {\mathbb G}_m^n\}$ 
 given by the identity map.
 Then $T$ is the same as the torification $T_\Delta$ if we consider ${\mathbb G}_m^n$ as toric variety with fan $\Delta=\{0\}$.
 Proposition \ref{prop_S_and_S_tilde_are_isomorphic} states that $\mathcal S({\mathbb G}_m^n,T)\simeq\mathcal S({\mathbb G}_m^n)$ and 
 Theorem \ref{thm_soule} says that $\mathcal S({\mathbb G}_m^n)$ is an S--variety such that $\mathcal S({\mathbb G}_m^n)_\Z\simeq{\mathbb G}_m^n$.
 
 In the second step, let $X$ be affine with torification $T$. 
 In this case, $X$ itself appears in the maximal torified atlas $\{U_i\}_{i\in I}$ of $X$, say $X=U_0$.
 Then $\mathcal S(X,T)=(\uuline X,\mathcal A_X,e_X)$ has the following simple description.
 Let $X_0=(\uline X_0,\mathcal A_0,e_0)$ be the S--gadget defined by $U_0$.
 For every affine S--variety $V$, we have $\uuline X(V)=\Hom(V,X_0)$, we have $\mathcal A_X=\mathcal A_0$ and
 $e_X(V)$ sends a morphism $\vphi=(\uline\vphi,\vphi_\C^\ast)\in\Hom(V,X_0)$ to $\vphi_\C^\ast\in\Hom(\mathcal A_0,\mathcal A_V)$.
 From this description it follows that we can apply \cite[Prop.\ 3]{Soule2004} to derive
 that $\mathcal S(X,T)$ is an S--variety if and only if $X_0$ is an affine S--variety,
 and if this is the case then $\mathcal S(X,T)_\Z\simeq (X_0)_\Z$.
 The same idea as used in the proof of Theorem \ref{thm_ll} applies to this situation.
 Namely, let $V$ be an affine S--variety and let $(\uline\vphi,\vphi_\C^\ast):X_0\to V$ be a morphism of S--gadgets.
 Every irreducible component of $X$ has a unique open subtorus isomorphic to $\mathbb G_m^n$ for some $n\geq0$ in the torification $T$.
 In the first step, we showed that $\mathcal S(\mathbb G_m^n)$ is an S--variety. Thus the S--gadget $(\mathbb G_m^n)_0$
 defined by $\mathbb G_m^n$ is an affine S--variety. 
 Using the universal property of $(\mathbb G_m^n)_0$ defines a rational map $\vphi_\Z:X\to V_\Z$.
 For the same reasons as in the proof of Theorem \ref{thm_ll} we see that $\vphi_\Z$
 is indeed a morphism of schemes that verifies the universal property of an affine S--variety for $X$.
 
 In the third and last step, we let $(X,T)$ be a general affinely torified variety with maximal torified atlas $\{U_i\}_{i\in I}$.
 Then $U_i$ is affine and $T$ restricts to a torification $T_i$ of $U_i$ for every $i\in I$.
 By the previous step $\mathcal S(U_i,T_i)$ is an S--variety such that $\mathcal S(U_i,T_i)_\Z\simeq U_i$.
 The family $\{\mathcal S(U_i,T_i)\}_{i\in I}$ satisfies the conditions of \cite[Prop.\ 5]{Soule2004}, 
 and thus $\mathcal S(X,T)$ is an S--variety with $\mathcal S(X,T)_\Z\simeq\bigcup_{i\in I}U_i\simeq X$.

 Concerning functoriality, it is clear that for all torified morphisms $f: (X,T)\to (X',T')$ between affinely torified varieties, the diagram
 $$ \xymatrix{X\ar[rr]^f\ar[d]^\sim&&X'\ar[d]^\sim\\\mathcal{S}(X,T)_\Z\ar[rr]^{\mathcal S(f)}&&\mathcal{S}(X',T')_\Z} $$
 commutes. This establishes an isomorphism between the composition of $\mathcal S$ and $-\otimes_\Fun\Z$ 
 and the forgetful functor from affinely torified varieties to schemes.
\end{proof}


\section{Deitmar's geometry}\label{section:deitmar}

\subsection{D--schemes}
 First let us recall the theory of schemes over $\F_1$ in Deitmar's sense. 
 The main idea is to substitute commutative rings with $1$ (called rings in the latter) 
 by commutative semi-groups with $1$ (called monoids in the latter) and to mimic scheme theory for monoids.
 It turns out that to a far extent, it is possible to obtain a theory that looks formally the same as usual algebraic geometry.
 Since definitions are lengthy, we only name the notions we make use of and give the reference to the proper definition
 in Deitmar's paper \cite{Deitmar2005}.

 There is the notion of prime ideals and the spectrum $\spec A$ of a monoid $A$ (\cite[section 1]{Deitmar2005}),
 schemes $X$ over $\F_1$ with underlying topological space $X^\top$ and morphisms of schemes (\cite[section 2.3]{Deitmar2005}), 
 the structure sheaf $\O_X$ and local monoids $\O_{X,x}$ for $x\in X^\top$ (\cite[sections 2.1--2.2]{Deitmar2005}).

 There is a base extension functor $-\otimes_{\mathbb F_1}\mathbb Z$ that sends $\spec A$ to $\Spec \mathbb Z[A]$,
 where $\mathbb Z[A]$ is the semi-group ring of $A$. The right-adjoint of $-\otimes_{\mathbb F_1}\mathbb Z$ 
 is the forgetful functor from rings to monoids (\cite[Thm. 1.1]{Deitmar2005}). 
 Both functors extend to functors between schemes over $\mathbb F_1$ and $\mathbb Z$ (\cite[section 2.3]{Deitmar2005}). 
 We will often write $X_\Z$ for $X\otimes_{\F_1}\Z$.
 We denote by \dtext{D--schemes} the category of schemes over $\F_1$ together with morphism of schemes in Deitmar's sense.

 A D--scheme $X$ is \dtext{connected} if it is connected as topological space. A D--scheme $X$ is \dtext{separated} if $X_\Z$ is separated. A monoid $A$ is \dtext{integral} if for every $a\in A$, the multiplication by $a$ defines an injective map $A\to A$. A D--scheme $X$ is \dtext{integral} (resp.\ \dtext{of finite type} resp.\ \dtext{of exponent $1$}) 
 if for all affine opens $\spec A$ of $X$, $A$ is integral
 (resp.\ $\mathbb Z[A]$ is of finite type (cf.\ \cite[Lemma 2]{Deitmar2006}) 
 resp.\ $1$ is the only element of finite multiplicative order in $A$).


\subsection{Toric varieties as D--schemes}
\label{toric_D-schemes}
 In \cite[section 4]{Deitmar2007},  Deitmar describes a functor $\mathcal D$ that associates to a toric variety $X$ with fan $\Delta$
 the following scheme over $\F_1$.
 Let $X$ be a toric variety with fan $\Delta$.
 We use the notation from section \ref{section_toric_as_torified}.
 An inclusion $\tau\subset\tau'$ of cones gives an inclusion of monoids $A_{\tau'}\subset A_\tau$ and thus 
 we obtain a directed system of affine D--schemes $\{\spec A_\tau\}_{\tau\in\Delta}$. 
 The D--scheme $\mathcal D(X)$ is defined as the limit over this system. It is immediate that $\mathcal D(X)$ is a connected separated integral D--scheme of finite type and exponent $1$.

 Let $(\psi,\tilde\psi,\{\psi_\tau\})$ be a toric morphism.
 The directed system of morphisms $\psi_\tau^\vee: A_{\tilde\psi(\tau)}\to A_\tau$
 describes a morphism $\mathcal D(f):\mathcal D(X)\to\mathcal D(X')$ of D--schemes.
 This establishes $\mathcal D$ as a functor.
 
 Every monoid $A$ has a unique maximal subgroup, namely 
 the group $A^\times$ of invertible elements, 
 and a unique maximal ideal, namely $\mathfrak m=A\setminus A^\times$.
 We define the \dtext{rank $\rk\tau$ of a cone $\tau$} as the rank of $A_\tau^\times$ and 
 the \dtext{rank $\rk x$ of a point $x$} in $X^\top$ as the rank of $\O_{X,x}^\times$.
 For every cone $\tau\in\Delta$, we have the canonical inclusion $\iota_\tau:\spec A_\tau\hookrightarrow\mathcal D(X)$. 
 We define 
 $$ \begin{array}{cccc} \Psi: & \Delta & \longrightarrow & \mathcal D(X)^\top \;,\!\!\! \\
                              & \tau   & \longmapsto     & \iota_\tau(\mathfrak m_\tau)  \end{array} $$
 where $\mathfrak m_\tau$ is the maximal ideal of $A_\tau$.

The following is a refinement of Deitmar's Theorem 4.1 in \cite{Deitmar2007}.

\begin{theorem}\label{thm_deitmar}\ 
 \begin{enumerate}
  \item \label{deitmar1} The functor $\mathcal D$ induces an equivalence of categories
        $$ \mathcal D:\ \ \left\{\begin{array}{c}\text{toric varieties}\end{array}\right\}
           \ \ \ \stackrel\sim\longrightarrow \ \ \ 
           \left\{\begin{array}{c}\textrm{connected separated integral D--schemes}\\ 
                                  \textrm{of finite type and of exponent }$1$ \end{array}\right\} $$
        with $-\otimes_{\mathbb F_1}\mathbb Z$ being its inverse.
  \item \label{deitmar2} Let $X$ be a toric variety with fan $\Delta$. Then $\Psi:\Delta\to\mathcal D(X)^\top$ is a bijection 
 such that $\tau\subset\tau'$ if and only if $\Psi(\tau')$ is contained in the closure of $\Psi(\tau)$. 
        Furthermore, $A_\tau\simeq\O_{\mathcal D(X),\Psi(\tau)}$ and $\rk\Psi(\tau)=\rk\tau$ for all $\tau\in\Delta$.
 \end{enumerate} 
\end{theorem}

\begin{proof}
 From the proof of \cite[Thm.\ 4.1]{Deitmar2007} it becomes clear that $Y\otimes_{\F_1}\Z$ is connected if $Y$ is a separated integral D--scheme of finite type that is connected and of exponent $1$. The rest of part \ref{deitmar1} of the theorem follows from \cite[section 4]{Deitmar2007}.

 We proceed with part \ref{deitmar2} of the theorem.
 First note that the assignment 
 $$ \Psi_1: \tau \ \mapsto \ (\Spec\Z[\tau^\vee]\hookrightarrow X) $$
 defines a bijection between
 $\Delta$ and the family of the affine opens $U$ of $X$ such that the inclusion $U\hookrightarrow X$ is a toric morphism.
 If $\tau\subset\tau'$ then $\Spec\Z[\tau^\vee]\subset\Spec\Z[\tau'^\vee]$.
 
 By the part \ref{deitmar1} of the theorem, the functor $\mathcal D$ puts this family in one-to-one correspondence 
 with the affine opens of $\mathcal D(X)$ and respects inclusions.
 
 Since $\O_{\spec A,\mathfrak m}=A$ if $A$ is a monoid with maximal ideal $\mathfrak m=A\setminus A^\times$ (cf.\ \cite[section 1.2]{Deitmar2005}),
 the assignment 
 $$ x \ \mapsto \ (\spec\O_{\mathcal D(X),x}\hookrightarrow\mathcal D(X)) $$ 
 defines a bijection between $\mathcal D(X)^\top$ and the affine opens of $\mathcal D(X)$.
 Note that $x$ is the image of the maximal ideal of $\O_{\mathcal D(X),x}$ 
 under the canonical inclusion $\spec\O_{\mathcal D(X),x}\hookrightarrow \mathcal D(X)$, which describes the inverse $\Psi_2$ of the latter bijection.
 If $x'$ is contained in the closure of $x$, then we have a inclusion $\spec\O_{\mathcal D(X),x}\to\spec\O_{\mathcal D(X),x'}$.
 Since $\Psi=\Psi_2\circ\mathcal D\circ\Psi_1$, we established that $\Psi$ is a bijection and
 that $\tau\subset\tau'$ if and only if $\Psi(\tau')$ is contained in the closure of $\Psi(\tau)$. 
 
 By definition, the rank of $\tau$ and the rank of $\Psi(\tau)$ equal the rank of the maximal subgroups of the monoids
 $A_\tau$ and $\O_{\mathcal D(X),\Psi(\tau)}$, respectively.
 The canonical inclusion $\spec A_\tau\hookrightarrow\mathcal D(X)$
 induces the isomorphism $\O_{\mathcal D(X),\Psi(\tau)}\simeq \O_{\spec A_\tau,\mathfrak m_\tau}\simeq A_\tau$,
 and consequently we obtain equality of ranks.
\end{proof}


\section{Comparison between the different geometries over $\F_1$}\label{section:geometries}

 In this section, we establish certain functors between the categories of D--schemes, S--objects and CC--gadgets
 and investigate to what extent they commute with base extension to $\Z$ and with the realizations of classes of varieties over $\F_1$
 from the previous sections. Finally, we put together the results of the paper in Theorem \ref{thm_large_diagram}.


\subsection{From D--schemes to CC--gadgets}
 In this section, we construct a functor $\FDCC$ from the category of integral D--schemes of finite type to the category of CC--gadgets.

 Let $X$ be an integral D--scheme of finite type. We define the CC--gadget $\FDCC(X)=(\uline X, X_\C, ev_X)$ as follows.
 For a finite abelian group $D$, we define $D_0$ to be the monoid $D\cup\{0\}$ that extends the multiplication of $D$ by $0\cdot a=0$ for every $a\in D$. 
 Put 
 $$ \uline X(D) \ = \ \Hom(\spec D_0,X) \ = \ \bigcup_{x\in X^\top}\, \Hom(O_{X,x}^\times,D)\;, $$
 where the latter equality is explained in the proof of Theorem 1 in \cite{Deitmar2006}.
 Put $X_\C=X\otimes_\Fun\C$, which is indeed a complex variety since the base extension of $X$ to $\C$ is a disjoint union of a toric varieties (cf.\ \cite[Thm.\ 4.1]{Deitmar2007}).
 Note that the immersion $\spec \mathcal O_{X,x}\hookrightarrow X$ induces $\Spec \C[\mathcal O_{X,x}]\hookrightarrow X_\C$
 and define $\ev_X(D)$ as the composition of the natural maps
 $$ \bigcup_{x\in X^\top}\! \Hom(O_{X,x}^\times,D) \ \longrightarrow \bigcup_{x\in X^\top}\! \Hom(\C[\mathcal O_{X,x}],\C[D])
              \ \longrightarrow \ \Hom(\Spec\C[D], X_\C) \;. $$

 Given a morphism $f:X\to X'$ between integral D--schemes of finite type and exponent $1$, we define $\FDCC(f)=(\uline f, f_\C)$,
 where $\uline f(D)=f_\ast:\Hom(\spec D_0,X)\to\Hom(\spec D_0,X')$ and $f_\C=f\otimes_\Fun\C:X_\C\to X'_\C$. 
 It is immediate that $\FDCC(f)$ is a morphism of CC--gadgets.
 
 Note that for a finite abelian group $D$, the set $\uline X(D)$ is finite. Putting 
 $$ \uline X^{(l)}(D) \ = \ \bigcup_{\substack{x\in X^\top\\ \rk x=l}} \Hom(O_{X,x}^\times,D) $$ 
 defines a grading $\uline X=\bigcup_{l\geq0} \uline X^{(l)}$.
 Thus we can consider $\FDCC(X)$ as a finite graded CC--gadget.
 
\begin{proposition}
 \label{prop_D_to_CC}
 The functors $\mathcal L$ and $\FDCC\circ\mathcal D$ from the category of toric varieties to the category of finite graded CC--varieties are isomorphic.
\end{proposition} 
 
\begin{proof}
 Let $X$ be a toric variety with fan $\Delta$ and put $Y=\mathcal D(X)$. 
 Then we obtain the finite graded CC--gadgets $\mathcal L(X)=(\uline X,X_\C,\ev_X)$ and $\FDCC(Y)=(\uline Y, Y_\C,\ev_Y)$.
 By part \ref{deitmar2} of Theorem \ref{thm_deitmar}, 
 there is a bijection $\Psi:\Delta\to Y^\top$ such that $\rk\tau=\rk\Psi(\tau)$ and $A_\tau\simeq\O_{Y,\Psi(\tau)}$ for every $\tau\in\Delta$.
 Thus we obtain for every $l\geq0$ and every finite abelian group $D$ a bijection
 $$ \uline X^{(l)}(D) \ = \ \bigcup_{\begin{subarray}{c}\tau\in\Delta\\ \rk\tau=l\end{subarray}} \Hom(A_\tau^\times, D) 
                 \ \ \stackrel\sim\longrightarrow \bigcup_{\begin{subarray}{c}y\in Y^\top\\ \rk y=l\end{subarray}} \Hom(\O_{Y,y}^\times,D) \ = \ \uline Y^{(l)}(D)\;.$$
 Further, $Y_\C=Y\otimes_\Fun\C\simeq X\otimes_\Z\C=X_\C$.
 It is immediate that these isomorphisms commute with the evaluation maps $\ev_X$ and $\ev_Y$,
 and we thus obtain the desired isomorphism of CC--gadgets $\vphi_X: \mathcal L(X) \to \FDCC(Y)$.
 
 It follows from the naturality of definitions that given a toric morphism $f:X\to X'$,
 the diagram
  \[ 
  \xymatrix{
   \mathcal L(X) \ar[rr]^{\mathcal L(f)} \ar[d]_{\vphi_X} & & \mathcal L(X') \ar[d]_{\vphi_{X'}} \\
   \FDCC(Y) \ar[rr]^{\FDCC(g)} & & \FDCC(Y')
  }
 \]
 commutes, where $Y=\mathcal D(X)$, $Y'=\mathcal D(Y')$ and $g=\mathcal D(f)$. Thus we established an isomorphism of functors.
\end{proof} 

 This proposition together with Theorems \ref{thm_ll} and \ref{thm_deitmar} implies: 

\begin{corollary}
 \label{cor_D_to_CC}
 If $X$ is a connected integral D--scheme of finite type and exponent $1$, 
 then $\FDCC(X)$ is a CC--variety and $\FDCC(X)\otimes_\Fun\Z\simeq X\otimes_\Fun\Z$.
\end{corollary}


\subsection{From D--schemes to S--objects}
 In this section, we construct a functor $\FDS$ from the category of D--schemes of finite type to the category of S--objects.
 Let $X$ be a D--scheme of finite type. We proceed in two steps, similarly to section \ref{toric_S-objects}.

 In the first step, we define for every point $x\in X^\top$ an S--gadget $X_x=(\uline X_x,\mathcal A_x,e_x)$ as follows.
 For every $R\in\mathcal R$, we put $\uline X_x(R)=\Hom(\mathcal O_{X,x},\mu(R)_0)$, 
 the set of monoid homomorphisms from the local monoid $\mathcal O_{X,x}$ 
 to the multiplicative monoid $\mu(R)_0=\{0\}\cup\mu(R)$,
 we put $\mathcal A_x=\C[\mathcal O_{X,x}]$, the semi-group ring of $\mathcal O_{X,x}$ over $\C$, 
 and we define 
 $$ e_x(R): \ \Hom(\O_{X,x},\mu(R)_0) \ \longrightarrow \ \Hom(\C[\O_{X,x}], R\otimes_\Z\C) $$
 as the natural map.

 In the second step, we define the object $\FDS(X)=(\uuline X, \mathcal A_X, e_X)$ as follows.
 For every affine S--variety $V$, we put 
 $$ \uuline X(V) \ = \ \bigcup_{x\in X^\top}\,\Hom(V, X_x) \;, $$
 where the union is taken in $\Hom(V_\Z,X_\Z)$. We put $\mathcal A_X=\O_{X_\C}(X_\C)$, where $X_\C$ is the complexification of $X_\Z$,
 and we define
 $$ e_X(V): \ \bigcup_{x\in X^\top}\,\Hom(V, X_x) \ \longrightarrow \ \Hom(V_\C, X_\C) \ \longrightarrow \ \Hom(\mathcal A_X,\mathcal A_V) $$
 as the composition of the natural maps.

 Given a morphism $f:X\to X'$ between D--schemes of finite type, there is a natural way to define a morphism $\FDS(f):\FDS(X)\to\FDS(X')$ going
 through the steps of the construction of $\FDS$, similarly to the definition in section \ref{toric_S-objects}.

\begin{proposition}
 \label{prop_D_to_S}
 The functors $\mathcal S$ and $\FDS\circ\mathcal D$ from the category of toric varieties to the category of S--objects are isomorphic.
\end{proposition} 

\begin{proof}
 Let $X$ be a toric variety with fan $\Delta$ and $Y=\mathcal D(X)$. 
 We will construct an isomorphism $\vphi_X:\mathcal S(X)\to\FDS(Y)$ by going through the steps of construction of the objects.
 
 In the first step, let $\tau\in\Delta$ and $y=\Psi(\tau)$, where $\Psi: \Delta\to Y^\top$ is the bijection from Theorem \ref{thm_deitmar}.
 Let $X_\tau$ and $Y_y$ be the associated S--gadgets. 
 By Theorem \ref{thm_deitmar}, part \ref{deitmar2}, we have that $A_\tau\simeq\O_{Y,y}$
 and consequently $\uline X_\tau(R)=\Hom(A_\tau,\mu(R)_0)\simeq\Hom(\O_{Y,y},\mu(R)_0)=\uline Y_y(R)$ for all $R\in\mathcal R$.
 Further, $\mathcal A_\tau=\C[A_\tau]\simeq\C[\O_{Y,y}]=\mathcal A_y$. 
 It is immediate that these isomorphisms commute with the evaluation maps $e_\tau$ and $e_y$, 
 and thus we obtain an isomorphism of S--gadgets $\vphi_\tau: X_\tau\to Y_y$.

 In the second step, we note that for $\tau'\subset\tau$, the image of the inclusion $A_\tau\hookrightarrow A_{\tau'}$ under the functor $\mathcal D$
 is the generalization map $\mathcal O_{Y,\Psi(\tau)}\hookrightarrow\mathcal O_{Y,\Psi(\tau')}$.
 Thus the directed systems $\{A_\tau\}_{\tau\in\Delta}$ and $\{\mathcal O_{Y,y}\}_{y\in Y^\top}$ are isomorphic
 and we have that for all affine S--varieties $V$,
 $$ \uuline X(V) \ = \bigcup_{\tau\in\Delta}\,\Hom(V,X_\tau) \ \simeq \ \bigcup_{y\in Y^\top}\,\Hom(V,Y_y) \ = \ \uuline Y(V) \;. $$
 Further, $\mathcal A_X=\O_{X_\C}(X_\C)\simeq\O_{Y_\C}(Y_\C)=\mathcal A_Y$ by Theorem \ref{thm_deitmar}, part \ref{deitmar1}.
 It is immediate that these isomorphisms commute with the evaluation maps $e_X$ and $e_Y$,
 and we thus obtain the desired isomorphism of S--objects $\vphi_X: \mathcal(X)\to \FDS(Y)$.

 By the analogy of the constructions of $\mathcal S$ and $\FDS$, it is clear that given a toric morphism $f:X\to X'$,
 the diagram
  \[ 
  \xymatrix{
   \mathcal S(X) \ar[rr]^{\mathcal S(f)} \ar[d]_{\vphi_X} && \mathcal S(X') \ar[d]_{\vphi_{X'}} \\
   \FDS(Y) \ar[rr]^{\FDS(g)} && \FDS(Y')
  }
 \]
 commutes, where $Y=\mathcal D(X)$, $Y'=\mathcal D(Y')$ and $g=\mathcal D(f)$. Thus we established an isomorphism of functors.
\end{proof}

 This proposition together with Theorems \ref{thm_soule} and \ref{thm_deitmar} implies: 

\begin{corollary}
 \label{cor_D_to_S}
 If $X$ is a connected integral D--scheme of finite type and exponent $1$, 
 then $\FDS(X)$ is an S--variety and $\FDS(X)\otimes_\Fun\Z\simeq X\otimes_\Fun\Z$.
\end{corollary}


\subsection{From CC--varieties to S--objects}
 \label{subsection:CC_to_S}
 In this section, we construct a functor $\FCCS$ from the category of CC--varieties to the category of S--objects. For this purpose, we have to restrict the class of morphisms between CC--varieties to those that satisfy property \textbf{S} described below.

 Let $X=(\uline X,X'_\C,\ev_X)$ be a CC--variety and let $\{X_j\}_{j\in J}$ be the family of all open affine CC--subvarieties
 $X_j=(\uline X_j, X_{j,\C},\ev_j)$ of $X$.
 Note that a priori, $X'_\C$ does not need to be equal to $X_\C=X\otimes_\Fun\C$.
 We define the S--object $\FCCS(X)$ in two steps.
 
 In the first step, we define S--gadgets $X_j^\sim=\FCCS^\sim(X_j)=(\uline X_j^\sim,\mathcal A_j^\sim,e_j^\sim)$ for every $j\in J$ as follows.
 For every $R\in\mathcal R$, put $\uline X_j^\sim(R)=\uline X_j(\mu(R))$.
 Put $\mathcal A_j^\sim=\O_{X_{j,\C}}(X_{j,\C})$ and put
 $$ e_j^\sim(R): \ \uline X_j(\mu(R)) \ \stackrel{\ev_j(\mu(R))}\longrightarrow \ \Hom(\mathcal A_j^\sim, \C[\mu(R)]) 
                 \ \longrightarrow \ \Hom(\mathcal A_j^\sim, R\otimes_\Z\C) \;. $$
 
 In the second step, we define the S--object $\FCCS(X)=(\uuline X,\mathcal A_X,e_X)$ as follows.
 For every $j\in J$ and every $R\in\mathcal R$, there is a morphism $\tau_j(R)$ given as the composition of canonical maps
 \[ 
  \ \uline X_j^\sim(R) \ = \ \uline X_j(\mu(R)) \subset \Hom(\Spec\Z[\mu(R)], X_{j,\Z}) \ \longrightarrow \ \Hom(\Spec R, X_{\Z}) \;.
 \]
 We do not know a priori whether $\tau_j(R)$ is an inclusion.
 But the morphism $(\vphi_j, \id): X_j^\sim\to\mathcal T(X_{j,\Z})$ yields---for the same reason as in Lemma \ref{lemma_univ_prop}---that 
 we have for every affine S--variety $V$ an (a priori not injective) map
 $$ \psi_{X_j}(V): \ \Hom(V,X_j^\sim) \ \longrightarrow \ \Hom(V_\Z, X_{j,\Z}) \ \subset \ \Hom(V_\Z,X_\Z) \;. $$
 Define $\uuline X(V)=\bigcup_{j\in J} \im\psi_{X_j}(V)\subset\Hom(V_\Z,X_\Z)$ and $\mathcal A_X=\O_{X_\C}(X_\C)$.
 Define
 $$ e_X(V): \ \uuline X(V) \ \subset \ \Hom(V_\Z,X_\Z) \ \longrightarrow \ \Hom(\mathcal A_X,\O_{V_\C}(V_\C))
         \ \longrightarrow \ \Hom(\mathcal A_X,\mathcal A_V) $$
 as the composition of taking complex global sections of a morphism $V_\Z\to X_\Z$ and the push forward 
 along the map $\iota_\C: \O_{V_\C}(V_\C)\hookrightarrow\mathcal A_V$ given by the universal property of $V$.

 Next we will define $\FCCS$ on morphisms. If $\varphi_j:X_j\to Y_j$ is a morphism between affine CC--varieties,
 then following through the definitions of the first step describes in a natural manner a morphism 
 $\varphi_j^\sim=\FCCS^{\sim}(\varphi_j): X_j^\sim\to Y_j^\sim$ of S--gadgets.
 In order to define a morphism $\FCCS(\varphi):\FCCS(X)\to\FCCS(Y)$ of S--objects associated to a gadget morphism $\varphi:X\to Y$ 
 between CC-varieties $X$ and $Y$, we need $\varphi$ to obey the following property. 

\begin{description}
  \item[S] \label{property:S} Let $\{X_j\}_{j\in J}$ be the family of all open affine CC--subvarieties whose images $Y_j$ under $\varphi$ is affine.
   Denote the corresponding restrictions of $\varphi$ by $\varphi_j:X_j\to Y_j$ 
        and let $\varphi_j^\sim:X_j^\sim\to Y_j^\sim$ be the associated morphism of S--gadgets.
        Let $V$ be an affine S--variety and let $\psi_{X_j}(V)$ be the morphism as defined above.
    Then there is a unique morphism $\overline\varphi_j$ such that the diagram
    \[
      \xymatrix{\Hom(V,X_j^\sim)\ar[rr]^{(\varphi_j^\sim)_\ast}\ar[d]^{\psi_{X_j}(V)}  && \Hom(V,Y_j^\sim) \ar[d]^{\psi_{Y_j}(V)} \\
              \im \psi_{X_j} \ar[rr]^{\overline\varphi_j} && \im \psi_{Y_j}} 
    \]
commutes.
\end{description}
 
 Given a morphism $\varphi:X\to Y$ of CC--varieties, we define $\FCCS(\varphi)=(\uuline\varphi,\varphi^\ast)$
 as follows. Let $V$ be an affine S--variety. With $\overline\varphi_j$ as in the definition, we define
 $$ \uuline\varphi(V)\ = \ \bigcup_{j\in J} \overline\varphi_j: \uuline X(V)\ = \ \bigcup_{j\in J}\; \im\psi_{X_j}(V) \ \longrightarrow \ 
    \bigcup_{j\in J}\; \im\psi_{X_j}(V) \ = \ \uuline Y(V) \;. $$
 Put $\varphi^\ast=\varphi_\C^\#:\mathcal O_{Y_\C}(Y_\C)\to\mathcal O_{X_\C}(X_\C)$.
 
\begin{remark}
 Property \textbf{S} is satisfied if for every $j\in J$, the map $\psi_{X_j}(V)$ is injective, 
 or if $X_j$ is an affine S--variety by applying the universal property of an S--variety.
 From the following proposition and from section \ref{affinely_torified_S-varieties}, 
 both the injectivity and the universal property is fulfilled if $\varphi$ is of the form $\varphi=\mathcal L(\varphi_\Z)$
 for a torified morphism $\varphi_\Z: (X_\Z,T)\to(Y_\Z,S)$ between affinely torified varieties.
 This means that the functor $\mathcal L$ is well-defined as functor from the category of affinely torified varieties
 to the \emph{category of CC--varieties} considered with the class of morphism from Definition \ref{def_morphism_of_CC-varieties}.
 Note further that the essential image of $\FDCC$ is contained in the category of CC--varieties since $\FDCC$ is essentially isomorphic
 to the composition of the base extension to toric varieties, considered as affinely torified varieties, and $\mathcal L$.
\end{remark}
 
 \begin{proposition} 
 \label{prop_CC_to_S}
 The functors $\mathcal S$ and $\FCCS\circ\mathcal L$ from the category of affinely torified varieties to the category of S--objects are isomorphic.
\end{proposition}

\begin{proof}
 Let $(X,T)$ be an affinely torified variety with maximal torified atlas $\{U_j\}_{j\in J}$.
 Let $T_j$ be the restriction of $T$ to $U_j$, which is a torification of $U_j$.
 Put $Y=(\uline Y,Y_\C,\ev_Y)=\mathcal L(X,T)$.
 Then $\{Y_j\}_{j\in J}$ with $Y_j=(\uline Y_j,Y_{j,\C},\ev_j)=\mathcal L(U_j,T_j)$ is the family
 of all open affine CC--subvarieties of $Y$ since they are precisely those open CC--subgadgets whose functors represent the right counting function.
 We show in two steps that $\mathcal S(X,T)\simeq\FCCS\circ\mathcal L(X,T)$.

 In the first step, we show that $X_j=(\uline X_j,\mathcal A_j,e_j)$ as defined in section \ref{affinely_torified_S-varieties}
 is isomorphic to $\FCCS^\sim(Y_j)=(\uline Y_j^\sim,\mathcal A_j^\sim,e_j^\sim)$ for every $j\in J$.
 For all $R\in\mathcal R$, we have equalities
 $$\uline X_j(R) \ = \ \coprod_{j\in T_j} \; \Hom(A_i^\times,\mu(R)) \ = \ \uline Y_j(\mu(R)) \ = \ \uline Y_j^\sim(R) $$
 and $\mathcal A_j=\O_{U_{j,\C}}(U_{j,\C})=\mathcal A_j^\sim$.
 This defines the desired isomorphism.

 In the second step, we show that $\mathcal S(X,T)=(\uuline X,\mathcal A_X,e_X)$ is isomorphic to the S-object $\FCCS(Y)=(\uuline Y,\mathcal A_Y,e_Y)$.
 For all affine S--varieties $V$, we have equalities
 $$ \uuline X(V) \ = \ \bigcup_{j\in J} \; \Hom(V,X_j) \ = \ \bigcup_{j\in J} \; \Hom(V,Y_j) \ = \ \uuline Y(V) $$
 and $\mathcal A_X=\O_{X_\C}(X_\C)=\mathcal A_Y$. 
 This defines the desired isomorphism, which we denote by $\vphi_{X,T}:\mathcal S(X,T)\to\FCCS\circ\mathcal L(X,T)$.
 
 By similarity of definition it follows that $\vphi_{X,T}$ is functorial in $(X,T)$, i.e.\ that for every torified morphism $f:(X,T)\to(X',T')$,
 the diagram
 \[ \xymatrix{\mathcal S(X,T) \ar[rrr]^{\mathcal S(f)} \ar[d]_{\vphi_{X,T}} &&& \mathcal S(X',T') \ar[d]_{\vphi_{X',T'}} \\
       \FCCS\circ\mathcal L(X,T) \ar[rrr]^{\FCCS\circ\mathcal L(f)} &&& \FCCS\circ\mathcal L(X',T')}  \]
 commutes. Thus we established an isomorphism of functors. 
\end{proof}

\begin{remark}
 \label{remark_image_of_FCCS}
 As consequence of Proposition \ref{prop_CC_to_S} and Theorem \ref{thm_javier}, we see that for every CC--variety $X$ in the essential image of $\mathcal L$,
 the S--object $\FCCS(X)$ is an S--variety such that $\FCCS(X)_\Z\simeq X_\Z$.
 It is, however, not clear if this holds true if $X$ is an arbitrary CC--variety.

 Namely, there are two problems. 
 For simplicity, we assume that $X$ is an affine CC--variety with canonical immersion $\iota: X\to\mathcal G(X_\Z)$.
 
 The first problem is the following. 
 We have $X_\Z=\Spec B$ for some ring $B$. Put $X^\sim=\FCCS^\sim(X)$. 
 Then there is a canonical morphism $\iota^\sim: X^\sim\to\mathcal T(X_\Z)$, but it is not clear if the map
 $$ \xymatrix{\uline\iota^\sim(R):\ \uline X^\sim(R) \ = \ \uline X(\mu(R))\ar@{^(->}[rr]^>>>>>>>>>{\uline\iota(\mu(R))}
               &&\Hom(B,\Z[(\mu(R)])\ar[r]&\Hom(B,R)} $$
 is injective for all $R\in\mathcal R$ 
 (here $\uline X$, $\uline X^\sim$, $\uline\iota$ and $\uline\iota^\sim$ denote the usual functors and natural transformations).
 
 The second problem is verifying the universal property of an S--variety. 
 This is, given a scheme $V$ of finite type over $\Z$ and a morphism of S--gadgets $\vphi:X^\sim\to\mathcal T(V)$,
 we seek a morphism of schemes $\vphi_\Z: X_\Z\to V$ such that $\vphi=\mathcal T(\vphi_\Z)\circ\iota^\sim$.
 This would be implied by the universal property for $X$ if we could extend the functor $\FCCS$ 
 to a functor $\FCCS'$ from CC--gadgets to S--objects such that 
 $$ \FCCS'\Biggl(\def\objectstyle{\scriptstyle}\def\labelstyle{\scriptstyle}\vcenter{
                 \xymatrix@-1pc{X\ar[r]^<<<\iota\ar[rd]_\psi&\mathcal G(X_\Z)\ar[d]^{\mathcal G(\psi_\Z)} \\ &\mathcal G(V)}}\Biggr) 
    \quad = \quad \Biggl(\def\objectstyle{\scriptstyle}\def\labelstyle{\scriptstyle}\vcenter{
               \xymatrix@-1pc{X^\sim\ar[r]^{\iota^\sim}\ar[rd]_\vphi&\mathcal G(X_\Z)\ar[d]^{\mathcal T(\psi_\Z)} \\ &\mathcal T(V)}}\Biggr)  $$
 for some morphism $\psi: X\to\mathcal G(V_\Z)$. The uniqueness of $\psi_\Z$ would follow from the existence of a left inverse functor to $\FCCS'$.

 However, the definition of $\FCCS$ relies strongly on the defining property of a CC--variety and we do not see whether there is a way to extend $\FCCS$
 to all CC--gadgets with the desired property. We will discuss two attempts in this direction in the following two paragraphs 
 \ref{subsubsection:CC_to_S} and \ref{subsubsection:S_to_CC}.
\end{remark}


\subsubsection{From CC--gadgets to S--objects}
\label{subsubsection:CC_to_S}
 There is a natural definition for a functor $\FCCS'$ from CC--gadgets to S--objects, 
 which, however, does not meet the requirements of Remark \ref{remark_image_of_FCCS}.
 
 Let $X=(\uline X, X_\C,\ev_X)$ be a CC--gadget. We define the S--object $\FCCS'(X)=(\uuline X,\mathcal A_X,e_X)$ as follows.
 If $V$ is an affine S--variety, where $V_\Z\simeq \Spec B$ and $(\uline \iota,\iota_\C^\ast):V\to \mathcal T(V_\Z)$ is the canonical immersion, 
 then put $\uuline X(V) = \uline X(\mu(B))$. 
 Put $\mathcal A_X=\O_{X_\C}(X_\C)$ and define for $\psi\in\uline X(\mu(B))$,
 \[
  \xymatrix{
   e_X(V)(\psi):\ \ \mathcal A_X \ \ \ar[rr]^{\ \ \ev_x(V)(\psi)^\#} & & \ \ \C[\mu(B)]\ \ \ar[r] & \ \ B\otimes_\Z\C\ \ \ar[r]^{\ \ \iota_\C^\ast}&\ \ \mathcal A_V\;.
  }
 \]
 If $\vphi=(\uline\vphi,\vphi_\C):X\to X'$ is a morphism of CC--gadgets, 
 define the morphism of S--objects $\FCCS'(\vphi)=(\uuline \vphi,\vphi_\C^\#)$ as follows.
 For $V$ as above, put $\uuline\vphi(V)=\uline\vphi(\mu(B))$ and
 let $\vphi_\C^\#$ be the morphism between global sections.
 One easily verifies that $(\uuline \vphi,\vphi_\C^\#)$ is indeed a morphism using that $(\uline\vphi,\vphi_\C)$ is one. 
 
\begin{remark}
 \label{remark_FCCS'_differs_from_FCCS}
 One can show that for a torified variety $(X,T)$ that is affine and has maximal torified atlas $\{U_i\}_{i\in I}$
 with $U_0=X$, the S--gadgets $X_0$ (as defined in section \ref{affinely_torified_S-varieties}) 
 and $\FCCS^\sim\circ\mathcal L(X,T)$ are isomorphic.
 Further, one can show that there is a natural inclusion of functors $\FCCS'\Rightarrow\FCCS$,
 when restricted to the category of CC--varieties.
 
 The most basic example of $X=\Gm$, however, shows that $\FCCS'$ is not isomorphic to $\FCCS$
 if restricted to the category of CC--varieties.
 Consider $\Gm$ as a toric variety with fan $\Delta=\{0\}$.
 In the usual notation (cf.\ sections \ref{section_toric_as_torified} and \ref{toric_S-objects}), 
 $A_0$ is an infinite cyclic group and $X_0=(\uline X_0,\mathcal A_X,e_X)$ is an affine S--variety with $(X_0)_\Z\simeq\Gm$.
 Let $\uuline  Y$ and $\uuline X$ be the functors of $\FCCS'\circ\mathcal L(X)$ and $\FCCS\circ\mathcal L(X)\simeq\mathcal S(X)$ 
 (cf.\ Proposition \ref{prop_CC_to_S}), respectively.
 Then 
 $$ \uuline Y(X_0) \ = \ \Hom(A_0,\mu(\Z[A_0])) \ = \ \Hom(A_0,\{\pm1\}) \ = \ \{\pm1\} \;. $$
 On the other hand, 
 $$ \uuline X(X_0) \ = \ \Hom(X_0,X_0) \ \hookrightarrow \ \Hom(\Gm,\Gm) \ = \ \Hom(\Z[A_0],\Z[A_0]) \ = \ \{\pm a^m\}_{m\in\Z}\;, $$ 
 where the inclusion is given by extension of scalars to $\Z$ (cf.\ Lemma \ref{lemma_univ_prop}).
 One sees that $\uuline Y(X_0)\subset \uuline X(X_0)$. We will show that this inclusion is proper. 

 Let $m$ be an integer and let $\vphi_m: A_0\to A_0$ map $a$ to $a^m$. 
 We define a morphism $\psi_m=(\uline \psi_m,\psi_{m,\C}):X_0\to X_0$ as follows.
 For $R\in\mathcal R$, we have $\uline X_0(R)=\Hom(A_0,\mu(R))$. Put
 $$\begin{array}{cccc}
     \uline\psi_m(R): & \Hom(A_0,\mu(R)) & \longrightarrow & \Hom(A_0,\mu(R)) \\
                    & \chi         & \longmapsto     & \chi\circ\vphi_m
    \end{array} $$
 and let $\psi_{m,\C}: \C[A_0]\to\C[A_0]$ be the $\C$-linear homomorphism that restricts to $\vphi_m$.
 It is clear that $\psi_m$ is indeed a morphism of S--gadgets for every $m\in\Z$
 and that $(\psi_m)_\Z^\#: \Z[A_0]\to\Z[A_0]$ is the restriction of $\psi_{m,\C}$ to $\Z[A_0]$.
 Concerning our question, we see now that $(\psi_m)_\Z^\#(A_0)\not\subset\mu(\Z[A_0])=\{\pm1\}$ unless $m=0$.
 
 Thus we have shown that $\FCCS'$ does not extend $\FCCS$. 
 From \cite[Prop.\ 4]{Soule2004} it follows that $\FCCS'(\Gm)$ cannot be an S--variety.
 Regarding the second problem of Remark \ref{remark_image_of_FCCS}, 
 note that it holds neither true that for a scheme $X$ of finite type over $\Z$, the S--objects 
 $\FCCS'(\mathcal G(X))$ and $\mathcal{O}b(X)$ are isomorphic. Namely, their functors $\uuline X'$ and $\uuline X$, respectively,
 differ. If $V$ is an affine S--variety with $V_\Z\simeq\Spec B$, then in general
 $$ \uuline X'(V) \ = \ \Hom(\Spec\Z[\mu(B)],X) \quad \neq \quad \Hom(\Spec B,X) \ = \ \uuline X(V) \;.$$
\end{remark}


\subsubsection{From S--objects to CC--gadgets}
\label{subsubsection:S_to_CC}
 There is also a natural way to define a functor $\FSCC$ from the category of S--objects to the category of CC--gadgets.

 Let $X=(\uuline X,\mathcal A_X,e_X)$ be an S--object. 
 Then we define the CC--gadget $\FSCC(X)=(\uline X, X_\C,\ev_X)$ as follows.
 For a finite abelian group $D$, put $V_D=\mathcal T(\Spec\Z[D])$, which is an affine S--variety by \cite[Prop.\ 2]{Soule2004}
 and since $\Z[D]\in\mathcal R$. Put $\uline X(D)=\uuline X(V_D)$.
 Let $\mathcal N_X$ be the nilradical of $\mathcal A_X$. Put $X_\C=\Spec(\mathcal A_X/\mathcal N_X)$, which is a complex variety.
 The evaluation map is defined as
 $$\begin{array}{cccl}
  \ev_X: & \uuline X(V_D) & \longrightarrow & \Hom(\mathcal A_X, \C[D]) \ = \ \Hom(\Spec\C[D],X_\C)\;. \\
         &    \psi        & \longmapsto     & \hspace{0,5cm} e_X(D)(\psi)
 \end{array}$$

\begin{remark}\label{remark:S_to_CC_problems}
 There are several remarks in order concerning the ``naturality'' of definition.
  Since we stay with the original definition of a CC--gadget in \cite{Connes2008},
  we only allow complex varieties, i.e.\ reduced schemes of finite type over $\C$, in the definition of a CC--gadget.
  Therefore, we have to divide out the nilradical. 
  One can, however, extend Connes-Consani's definition by allowing arbitrary schemes of finite type over $\C$
  and simply define $X_\C$ as the spectrum of $\mathcal A_X$.  

  If $X$ is an S--variety representing a scheme that is not affine, we obtain a complex variety $X_\C$ which is affine.
  One could, however, exchange the complex algebra by a scheme of finite type over $\C$ in the definitions of an S--gadget and an S--object,
  and try to recover the results of Soul\'e's paper \cite{Soule2004}. 
  Then one could simply define to take the same complex scheme for $\FSCC(X)$.
\end{remark}

\begin{remark}
 Unfortunately, the different nature of Soul\'e's and Connes-Consani's geometries over $\Fun$ leads to a misbehavior of $\FSCC$ 
 even if the suggested changes are made, as can be seen in the example of $X=\Gm$.

 In the same notation as in Remark \ref{remark_FCCS'_differs_from_FCCS},
 let $A_0$ be the infinite cyclic group and $X_0$ the affine S--variety associated to $X$.
 Let $\uuline X$ be the functor of $\mathcal S(X)$ and let $\uline X$ be the functor of $\FSCC\circ\mathcal S(X)$.
 For a finite cyclic group $D$ and $V_D$ as above, we have $\uline X(D)=\uuline X(V_D)=\Hom(V_D,X_0)$.
 Base extension from $\F_1$ to $\Z$ defines the inclusion $\Hom(V_D,X_0)\hookrightarrow\Hom(\Z[A_0],\Z[D])$ (cf.\ Lemma \ref{lemma_univ_prop}).
 Using that $\mu(\Z[D])=\Z[D]^\times$ for finite abelian groups, 
 one can show that conversely every morphism $\Z[A_0]\to\Z[D]$ defines a morphism $V_D\to X_0$.
 Thus we see that 
 $$ \uline X(D) \ = \ \Hom(A_0,\mu(\Z[D])) \ = \ \mu(\Z[D]) \ = \ D\amalg -D \;. $$
 This differs from the CC--variety $\mathcal L(\Gm)=(\uline\Gm,\mathbb G_{m,\C},\ev_{\Gm})$ since $\uline\Gm(D)=D$,
 and we see that $\mathcal L$ and $\FSCC\circ\mathcal S$ are not isomorphic.
 Furthermore, the counting function of $\FSCC(X)$ differs from the counting function of $\mathcal L(\Gm)$,
 so $\FSCC(X)$ is not even a candidate for a CC--variety representing $\Gm$ that produces the right counting function.
 
 In particular, one verifies now easily that neither $\FCCS'\circ\FSCC$ nor $\FSCC\circ\FCCS'$ nor $\FSCC\circ\FCCS$ 
 is isomorphic to the identity functor
 of the category of S--objects or the category of CC--gadgets, respectively--even if the changes are considered as suggested in the previous remark.
\end{remark}


\subsection{Putting pieces together}

 Finally, we subsume the results of this section in a diagram. In this section we consider only morphisms that satisfy property \textbf{S} from section \ref{subsection:CC_to_S} in the category of CC--varieties.

\nopagebreak{
\begin{theorem}
 \label{thm_large_diagram}
 The following diagram commutes up to natural isomorphism of functors
 (arrows with label ``$i$'' are the canonical inclusion as subcategories
  and the arrow with label ``$f$'' is the forgetful functor).
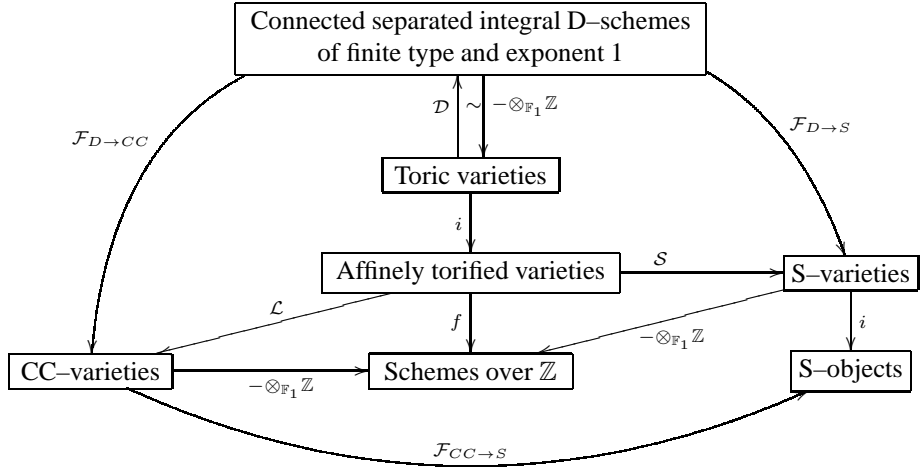
\[
 \xy (115,10)*+[F]{\text{\rm S--objects}}="B";
  (65,10)*+[F]{\begin{minipage}{7em}\rm \center Schemes over $\Z$ \end{minipage}}="C";
  (65,54)*+[F]{\begin{minipage}{17em}\rm \center Connected separated integral D--schemes\\ of finite type and exponent 1 \end{minipage}}="D";
  (65,36)*+[F]{\begin{minipage}{6em}\rm \center Toric varieties\end{minipage}}="E";
  (115,23)*+[F]{\text{\rm S--varieties}}="G";
  (65,23)*+[F]{\begin{minipage}{10.5em}\rm \center Affinely torified varieties \end{minipage}}="I";  
  (15,10)*+[F]{\begin{minipage}{5.5em}\rm \center CC--varieties\end{minipage}}="J";  
  {\ar@/_4pc/_>>>>>>>>>>>>>>>>>>>>>>>>{\mathcal{F}_{D\to CC}} "D"; "J"};
  {\ar@/^3pc/^<<<<<<<<<<{\mathcal{F}_{D\to S}} "D"; "G"};
  {\ar@<1.1ex>_>>>>>{\sim}^>>>>>{\mathcal{D}} "E"; "D"};
  {\ar@<1.1ex>^<<<<<{-\otf \Z} "D"; "E"};
  {\ar_{i} "E"; "I"};
  {\ar_{f} "I"; "C"};
  {\ar^{\mathcal{S}} "I"; "G"};
  {\ar_{\mathcal{L}} "I"; "J"};
  {\ar^{i} "G"; "B"};
  {\ar^{-\otf\Z} "G"; "C"};
  {\ar_{-\otf\Z} "J"; "C"};
  {\ar@/_3pc/^{\mathcal{F}_{CC\to S}} "J"; "B"};
 \endxy
\]
\end{theorem}}

\begin{proof}
 We label the subdiagrams as follows.
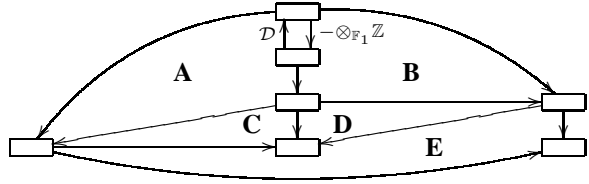
\[\xy (70,6)*+[F]{\quad}="B";
  (35,6)*+[F]{\quad}="C";
  (35,24)*+[F]{\quad}="D";
  (35,18)*+[F]{\quad}="E";
  (70,12)*+[F]{\quad}="G";
  (35,12)*+[F]{\quad}="I";  
  (0,6)*+[F]{\quad}="J";  
  {\ar@/_1.2pc/ "D"; "J"};
  {\ar@/^1pc/ "D"; "G"};
  {\ar@<1.1ex>^{\mathcal{D}} "E"; "D"};
		{\ar@<1.1ex>^{-\otf \Z} "D"; "E"};
		{\ar "E"; "I"};
  {\ar "I"; "C"};
  {\ar "I"; "G"};
  {\ar "I"; "J"};
  {\ar "G"; "B"};
  {\ar "G"; "C"};
  {\ar "J"; "C"};
  {\ar@/_1pc/ "J"; "B"};
  (20,16)*{\text{\textbf{A}}};
  (50,16)*{\text{\textbf{B}}};
  (29,9)*{\text{\textbf{C}}};
  (41,9)*{\text{\textbf{D}}};
  (53,6)*{\text{\textbf{E}}};                      \endxy\]
 The functors $\mathcal D$ and $-\otimes_\Fun\Z$ are mutually inverse by Theorem \ref{thm_deitmar}.
 Subdiagram A commutes (up to isomorphism, as all the commutations mentioned below) by Proposition \ref{prop_D_to_CC} and Corollary \ref{cor_D_to_CC}.
 Subdiagram B commutes by Proposition \ref{prop_D_to_S} and Corollary \ref{cor_D_to_S}.
 Subdiagram C commutes by Theorem \ref{thm_ll}.
 Subdiagram D commutes by Theorem \ref{thm_soule}.
 Subdiagram E commutes with the rest of the diagram by Theorem \ref{prop_CC_to_S}.
\end{proof}


\section{Concluding Remarks}

\subsection{On Chevalley schemes over $\F_1$}
 Among other reasons, Tits' suggestion
 of realizing Chevalley schemes as group objects over $\F_1$ (\cite[section 13]{Tits1957})
 was a main motivation in looking for concepts of geometries
 that have a base extension functor to $\Z$ and that somehow capture the aspects of usual geometry that can be ``expressed by roots of unity''.
 We discuss in various examples in how far Tits' suggestion becomes realized 
 by the different concepts of Connes-Consani, Soul\'e and Deitmar, respectively.
 
 To realize a Chevalley scheme $G$ as a group object in one of the discussed notions of geometries over $\F_1$ means that
 there is a CC--variety, an S--variety or a D--scheme $X$, respectively, representing $G$
 and a multiplication map $m: X\times X\to X$ such that 
 $X_\Z$ together with $m_\Z$ is an algebraic group isomorphic to $G$.
 In this case we say that $X$ together with $m$ is a group object over $\F_1$.
 
\begin{proposition}
 For every $n\geq0$, the Chevalley schemes $\mathbb G_m^n$ can be realized as group objects over $\F_1$ in all three notions of geometry over $\F_1$.
\end{proposition}

\begin{proof}
 The crucial observation is that the multiplication $\mathbb G_m^n\times\mathbb G_m^n\to\mathbb G_m^n$ is a toric morphism.
 With this, Theorem \ref{thm_ll} implies that $\mathcal L(\mathbb G_m^n)$ together with $\mathcal L(m)$ is a group object over $\F_1$.
 Theorem \ref{thm_javier} implies that $\mathcal S(\mathbb G_m^n)$ together with $\mathcal S(m)$ is a group object over $\F_1$.
 Theorem \ref{thm_deitmar} implies that $\mathcal D(\mathbb G_m^n)$ together with $\mathcal D(m)$ is a group object over $\F_1$.
\end{proof}

\begin{proposition}
 \label{prop_no_additive_group_over_F_1}
 For every $n>0$, the algebraic group $\mathbb G_a^n$ cannot be realized as group object in any of the three notions of geometries over $\F_1$.
\end{proposition}

\begin{proof}
 First, we consider Connes-Consani's concept. 
 Assume there was a group object $X=(\uline X,X_\C,\ev_X)$ with multiplication $m$ representing $\mathbb G_a^n$.
 We first want to exclude the possibility that the image of $\ev_X(D):\uline X(D)\to\Hom(\Spec\C[D],\mathbb G_a^n)$ 
 consists of only one element for all finite abelian groups $D$.
 If this was the case, then the image of $\ev_X(D)$ 
 would consist of the same point $x\in\mathbb G_a^n(\C)$ for all finite abelian groups $D$ by the functoriality of $\uline X$.
 But then the composition $\iota\circ\vphi$ of an automorphism $\vphi:X\to X$ given by a morphism $X_\C\to X_\C$ that leaves $x$ fixed
 but is not defined over $\Z$ followed by the canonical immersion $\iota:X\to\mathcal G(\mathbb G_a^n)$ would be a morphism of CC--gadgets
 that does not base extend to $\Z$. 

 Thus assume that $D$ is a group such that the image of $\ev_X(D)$ has more than one element.
 Then the commutative diagram
 $$ \xymatrix{\uline X(D)\times \uline X(D) \ar[rrr]^{m(D)} \ar[d]_{\ev_X(D)\times\ev_X(D)} &&&\uline X(D)\ar[d]^{\ev_X(D)}\\ 
               \mathbb G_a^n(\C[D])\times \mathbb G_a^n(\C[D])\ar[rrr]^{m_\C(\C[D])}&&& \mathbb G_a^n(\C[D])} $$
 would establish the image of $\ev_X(D)$ as a non-trivial finite subgroup of the torsion free group $\A^n(\C[D])\simeq\C^{nd}$ where $d=\# D$,
 which does not exist.
 Thus we showed that $X$ and $m$ as assumed cannot exist.
 
 A similar argument shows that $\mathbb G_a^n$ cannot be realized in Soul\'e's geometry over $\F_1$.

 Since, up to isomorphism, the only D--scheme representing $\A^n$ is $Y=\mathcal D(\mathbb G_a^n)$,
 the existence of a multiplication of $Y$ would imply by Theorem \ref{thm_deitmar}, that the multiplication
 of $\mathbb G_a^n$ is a toric morphism, which is not the case.
\end{proof}

\subsubsection{Chevalley groups as CC--varieties}
 \label{remark_CC_result_on_group}
 In their paper \cite{Connes2008}, Connes and Consani show that a split Chevalley scheme $G$ over $\Z$ is ``a variety over $\F_{1^2}$'' 
 (\cite[Thm.\ 4.10]{Connes2008}) and they remark that the normalizer $N$ of a maximal split torus $T$ in $G$ is a group object over $\F_{1^2}$,
 but that the multiplication of $G$ is ``more mysterious'' (ibid.\ 25).
 The following example shows that neither the multiplication of $G$ nor the multiplication of $N$ has to be defined over $\F_1$.

 Let $G=\Sl(2)$. 
 Let $T$ be the diagonal torus, $N$ its normalizer in $G$ and $B$ the subgroup of upper triangular matrices.
 We saw in Example \ref{example:Sl(2)} that we have torifications
 $$ N \ = \ 2\Gm \quad \quad \subset \quad \quad G \ = \ 2\Gm \ \du \ 3\Gm^2 \ \du \ \Gm^3 \;. $$
 Write $S$ for the torification of $G$ and by $S'$ the restriction of $S$ to $N$.
 Let $X=(\uline X,X_\C,\ev_X)$ be $\mathcal L(G,S)$ and let $Y=(\uline Y,Y_\C,\ev_Y)$ be $\mathcal L(N,{S'})$.
 Then 
 $$ \uline Y(D) \ = \ 2D \quad \quad \subset \quad \quad \uline X(D) \ = \ 2D \ \du \ 3D^2 \ \du \ D^3 $$   
 for a finite abelian group $D$. 
 Note that a multiplication of $X$ restricts to a multiplication of $Y$,
 and thus we only have to show the non-existence of a multiplication for $Y$.
 Assume there is  a multiplication $m: Y\times Y\to Y$, then for the trivial group $D=\{0\}$, 
 we can identify $\uline Y(\{0\})$ with $W$, and $\ev(\{0\}): W\to Y_\C(\C)=N(\C)$ defines a section to 
 $$ \xymatrix{1\ar[r]&T(\C)\ar[r]&N(\C)\ar[r]&W\ar[r]&1} \;. $$ 
 Moreover, the commutative diagram
 $$ \xymatrix{W\times W \ar[rr]^{m(\{0\})} \ar[d]_{\ev_Y(\{0\})\times\ev_Y(\{0\})} &&W\ar[d]^{\ev_Y(\{0\})}\\ N(\C)\times N(\C)\ar[rr]^{m_\C}&&N(\C)} $$
 that we obtain from the definition of a morphism between CC--gadgets 
 implies that the section $W\to N(\C)$ must be a group homomorphism. But this is not possible in the case of $\Sl(2)$.

\subsubsection{Chevalley groups as S--varieties}
 \label{remark_S_result_on_group}
 The situation in Soul\'e's geometry behaves similarly except for one remarkable difference.
 Since all rings $R\in\mathcal R$ are by definition flat over $\Z$, their additive groups are torsionfree
 and the group morphism $\mu(\Z)\to\mu(R)$ is thus injective. This means that $\mu(R)$ has a distinguished element of order $2$,
 namely, the image of $-1\in\mu(\Z)$.
 This allows us to transfer the idea of Connes-Consani, 
 which is to consider Chevalley schemes over $\F_{1^2}$ (see previous remark and \cite[section 4]{Connes2008}),
 to show that the normalizer $N$ of a maximal split torus $T$
 in a split Chevalley scheme $G$ is a group object in Soul\'e's notion of a geometry over $\F_1$.

 But there is no larger subgroup of $G$ than $N$ that can be realized as a group object in Soul\'e's geometry
 since this would involve additive structure.
 The argument of Proposition \ref{prop_no_additive_group_over_F_1} shows that this is not possible
 as it is not in the situation of Connes-Consani's paper (loc. cit.).

\begin{remark}
 A possible way out of the dilemma could be to broaden the notion of a morphism in Connes-Consani's or Soul\'e's geometry over $\F_1$.
 This could possibly be done by a motivic theory over $\F_1$ as already motivated in \cite{Manin1995}.
\end{remark}
 
\subsubsection{Chevalley groups as D--schemes}
 \label{remark_D_result_on_group}
 A Chevalley scheme can be realized in Deitmar's notion of a geometry over $\F_1$ 
 if and only if the Chevalley scheme is a toric variety and the multiplication is a toric morphism.
 This class of Chevalley schemes is precisely the class of split tori.

\subsection{Odds and ends}
 As we have noted in Remarks \ref{remark_torified_gadgets} and \ref{remark_torified_objects},
 different (affine) torification can lead to non-isomorphic CC--gadgets or S--objects, respectively.
 One may put the question: 
 shall it be an essential feature of a geometry over $\F_1$ to obtain different forms of a torified variety by choosing different torifications?
 There are two possible approaches to avoid the ambiguity of a torification:
 weakening the notion of morphism to gain isomorphic CC--varieties by different choices of torifications or
 using the following notion. 
 We call a decomposition $X=\decomp_{i\in I} Y_i$ \dtext{regular} if for every $i\in I$ there exists $J_i\subseteq I$ 
 such that $\overline{Y_i}=\decomp_{j\in J_i} Y_j$. 
 In other words, the Zariski closure of each of the schemes in the decomposition decomposes through the same decomposition.  
 Whenever a torified variety $X$ has a regular torification and any two regular torifications lead to isomorphic CC--varieties, 
 then one can declare the corresponding isomorphism class of CC--varieties as the \emph{canonical model of $X$ over $\F_1$}.
 Note that split tori, affine space, projective space and flag varieties have a unique isomorphism class of regular torifications.
 We do not know whether this is the case for all torified varieties.
 
 A second matter is the problem of the realization of the Grassmannian $\Gr(2,4)$ over $\F_1$ as posed by Soul\'e (\cite[Question 3]{Soule2004}),
 which stays open. It is not at all clear to us what this should be in Soul\'e's geometry over $\F_1$.
 Concerning Connes-Consani's notion, we present in this paper the candidate $\mathcal L(\Gr(2,4),T)$, 
 where $T$ is a torification given by a Schubert cell decomposition.
 Since, however, $T$ is not an affine torification, this CC--gadget fails to be a CC--variety.
 A possible solution could be searched in relaxing the notion of a CC--variety in an appropriate way.
 
 Note that the idea of establishing affinely torified varieties $(X,T)$ as varieties over $\F_1$ is quite flexible.
 We showed that it works in both Soul\'e's definition and Connes-Consani's definition.
 It further works with the modifications recently suggested by Connes and Consani in the end of their paper \cite{Connes2008}:
 There is a natural extension of the functors from finite abelian groups to monoids with distinguished elements $0$ and $1$
 since the CC--gadgets of torified varieties is defined in terms of homomorphism sets $\Hom(A_i,-)$, where the $A_i$ are free abelian groups.
 First note that it is not essential for our construction that we restrict $\uline X$ to finite abelian groups, but we can allow arbitrary abelian groups.
 Secondly, every homomorphism from a group into a monoid factorizes through the group of invertible elements of the monoid.
 Further, one might exchange the complex variety by a functor on rings that yields a reduced scheme of finite type over any ring.
 Namely, the result \cite[Thm.\ 5.1]{Connes2008} holds true for affinely torified varieties due to Lemma \ref{lemma:decomp}:
 there is a natural definition of evaluations $\ev_{X,A}: \uline X \Rightarrow X_A(A[-])$ for every ring $A$ and $X_A=X\otimes_\Z A$.
 If $A$ is a field and $M$ its multiplicative monoid, then 
 $$ \uline X(M) \ \stackrel{\ev_{X,A}(M)}\longrightarrow \ X_A(A[M]) \ \longrightarrow X_A(A) $$ 
 is a bijection, where the latter morphism is induced by the $A$-linear map $A[M]\to A$ identifying $M$ with $A$.


\end{document}